\documentclass[a4paper,10pt,leqno]{amsart}
\usepackage[margin=3cm]{geometry}
\usepackage{enumerate, amsmath, amsfonts, amssymb, amsthm, mathtools, thmtools,xparse,array,colortbl,hyperref,url,hypcap,tikz,xspace,accents,enumitem,letltxmacro,comment,bbm,ytableau,stmaryrd,bm}
\usepackage{nicematrix}
\hypersetup{colorlinks=true, citecolor=darkred, linkcolor=darkred,urlcolor=darkred}
\usepackage[procnames]{listings}
\usepackage[colorinlistoftodos]{todonotes}
\usetikzlibrary{calc,through,backgrounds,shapes,matrix,math,trees}
\definecolor{darkblue}{rgb}{0.0,0.7,0} 
\definecolor{darkred}{rgb}{.58,0,.83}%{0.7,0,0} 
\newcommand{\darkred}{\color{darkred}}

\numberwithin{equation}{section}
\newtheorem{theorem}{Theorem}[section]
\newtheorem{proposition}[theorem]{Proposition}
\newtheorem{corollary}[theorem]{Corollary}
\newtheorem{lemma}[theorem]{Lemma}
\newtheorem{conjecture}[theorem]{Conjecture}

\theoremstyle{definition}
\newtheorem{definition}[theorem]{Definition}

\newtheorem{example}[theorem]{Example}
\theoremstyle{remark}
\newtheorem{remark}[theorem]{Remark}

\usepackage{tikz}
\usetikzlibrary{math}
\usetikzlibrary{knots}

\usepackage[capitalize,nameinlink]{cleveref}
\crefformat{footnote}{#2\footnotemark[#1]#3}
\crefformat{conjecture}{Conjecture~#2#1#3}
\crefname{figure}{Figure}{Figures}
\crefformat{openproblem}{Problem~#2#1#3}
\crefformat{problem}{Problem~#2#1#3}

\renewcommand{\t}{\mathsf{t}}
\newcommand{\w}{\mathsf{w}}
\renewcommand{\u}{\mathsf{u}}

\newcommand{\T}{\textsc{tree}}

\newcommand{\F}{\textsc{fact}}
\newcommand{\Ft}{\textsc{f}\widetilde{\textsc{a}}\textsc{ct}}
\renewcommand{\S}{\textsc{sub}}
\newcommand{\E}{\T}

\newcommand{\FF}{\mathbb{F}}
\newcommand{\Fact}{\textsc{fact}}
\newcommand{\Inv}{\textsc{inv}}
\newcommand{\AS}{\widetilde{S}}

\newcommand{\hw}{\widehat{w}}

\newcommand{\EE}{\underline{\E}}
\newcommand{\EFt}{\underline{\Ft}}
\newcommand{\rrT}{\rr^T}

\newcommand{\tr}{\mathrm{tr}}
\newcommand{\ww}{\widetilde{w}}

\newcommand{\rank}{r}

\renewcommand{\DH}{\mathrm{DH}}
\renewcommand{\mod}{\;\mathrm{mod}\;}

\newcommand{\rr}{\mathsf{r}}
\newcommand{\rru}{\rr^\u}
\newcommand{\Trr}{T^\rr}
\newcommand{\rrk}[1][k]{\mathrm{nb}(\rr,#1)}
\newcommand{\Cat}{\mathrm{Cat}}
\newcommand{\rrs}{\rr^k_1}
\newcommand{\rrf}{\rr^k_2}
\newcommand{\rrr}{r}
\newcommand{\blambda}{\bm{\lambda}}

\newlist{thmlist}{enumerate}{1}
\setlist[thmlist]{label=(\arabic{thmlisti}), ref=\thetheorem(\arabic{thmlisti}),noitemsep}
\addtotheorempostheadhook[theorem]{\crefalias{thmlisti}{theorem}}

\newcommand{\defn}[1]{\emph{\darkred #1}}
\usepackage[T1]{fontenc}

\newcommand{\crefitem}[2]{%
  \hyperref[#2]{\namecref{#1}~\labelcref*{#1}~\ref*{#2}}%
}
\renewcommand{\arraystretch}{1.2}
\newcommand\cc{\cellcolor{darkblue!20}}
\newcommand\cd{\cellcolor{darkred!20}}
\newcommand\ten{\raisebox{.15\height}{\scalebox{.8}{10}}}

\usepackage{xargs}
\usepackage{xstring}
\newcommandx{\polygon}[7][7=1.25]{
  \node[circle,minimum size=#5cm] (#2) at  #1 {};

  \foreach \t in {1,...,#3} {
    \coordinate (#2\t) at ($#1+(90-\t*360/#3:#4)$);
  }
  \draw[thin,black,fill=white,opacity=0.3,densely dashed] #1 circle (#4);
  \setcounter{intege}{1}
  \pgfmathsetcounter{intege}{1}
  \foreach \object in {#6}{
    \ifthenelse{\not\equal{\object}{}}{
      \filldraw[black] ($#1+($(90-\theintege*360/#3:#4)$)$) circle(2pt);
    }{
    }
    \node[inner sep=0pt] at ($#1+($#7*(90-\theintege*360/#3:#4)$)$) {\small$\object$};
    \pgfmathsetcounter{intege}{\theintege+1}
    \setcounter{intege}{\theintege}
  }
}

\newcounter{intege}

\title[An elaborate new proof of Cayley's formula]{An elaborate new proof of Cayley's formula}%[Counting Trees using Affine Braid Varieties]{Counting Trees using Affine Braid Varieties}

\author[Banaian]{Esther Banaian}
\address[Banaian]{Aarhus University}\email{banaian@math.au.dk}
\author[Hoang]{Anh Trong Nam Hoang}
\address[Hoang]{University of Minnesota}\email{hoang278@umn.edu}
\author[Kelley]{Elizabeth Kelley}
\address[Kelley]{University of Illinois Urbana Champaign}\email{kelleye@illinois.edu}
\author[Miller]{Weston Miller}
\address[Miller]{University of Texas at Dallas}
\email{weston.miller@utdallas.edu}
\author[Stack]{Jason Stack}
\address[Stack]{University of Texas at Dallas}
\email{jason.stack@utdallas.edu}
\author[Stephen]{Carolyn Stephen}
\address[Stephen]{University of Minnesota}\email{csteph@umn.edu}
\author[Williams]{Nathan Williams}
\address[Williams]{University of Texas at Dallas}
\email{nathan.williams1@utdallas.edu}

\begin{document}

\begin{abstract}
We construct a bijection between certain Deodhar components of a braid variety constructed from an affine Kac-Moody group of type $A_{n-1}$ and vertex-labeled trees on $n$ vertices.   By an argument of Galashin, Lam, and Williams using Opdam's trace formula in the affine Hecke algebra and an identity due to Haglund, we obtain an elaborate new proof for the enumeration of the number of vertex-labeled trees on $n$ vertices.
\end{abstract}

\maketitle

%%%%%%%%%%%%%%%%%%%%%%%%%%%%%%%%%%%%%%%%%%%%%%%%%%%%%%%%%%%%%%%%%
\section{Introduction}
\label{sec:introduction}
%%%%%%%%%%%%%%%%%%%%%%%%%%%%%%%%%%%%%%%%%%%%%%%%%%%%%%%%%%%%%%%%%
\subsection{Introduction}
\label{sec:history}

It is well-known that the following sets have size $n^{n-2}$:
\begin{itemize}
 \item $\T_n$, the set of vertex-labeled trees with $n$ vertices~\cite{cayley} (\defn{Cayley's formula}); and
 \item $\F_n$, the set of factorizations of the long cycle $(1,2,\ldots,n)$ in the symmetric group $S_n$ into a product of $(n-1)$ transpositions~\cite{goulden1997transitive}.
\end{itemize}

%The 16 vertex-labeled trees on $4$ vertices are illustrated in~\Cref{fig:trees} (ignore the arrows, green edges and labels, and products for now).

But finding a bijection between $\T_n$ and $\F_n$ is surprisingly tricky (for a discussion, we refer the reader to the excellent paper~\cite{goulden2002tree}; see also~\cite{stanley1997parking}).

It turns out to be much easier to show that \[(n-1)!|\T_n|=(n-1)!|\F_n|.\] We quickly sketch the bijection.  The factor $(n-1)!$ on the left-hand side comes from labeling the $n-1$ edges of a vertex-labeled tree bijectively with the numbers $[n-1]\coloneqq\{1,2,\ldots,n-1\}$.  Recording the edges in order of increasing edge-label---where the edge between vertex $i$ and vertex $j$ is recorded as the transposition $(i,j)$---gives a bijection between vertex- and edge-labeled trees and factorizations of \emph{all} $(n-1)!$ long cycles in $S_n$ into $(n-1)$ transpositions.  The tricky bit is to to identify which vertex- and edge-labeled trees have image in the original set $\F_n$ (the answer relies on a certain embedding).

In this paper, we consider related problems in the \emph{affine} symmetric group $\AS_n$.

\subsection{The affine symmetric group}
The \defn{affine symmetric group} $\AS_{n}$ can be thought of as the group of bijections $\ww: \mathbb{Z} \to \mathbb{Z}$ such that~\cite[Chapter 8]{bjorner2005combinatorics}
\begin{equation}
\ww(i+n)=\ww(i)+n \text{ and } \sum_{i=1}^n \ww(i) = \binom{n+1}{2}.
\end{equation}
Recall that the reflections of $\AS_n$ swap $i$ and $j$ for $i,j \in \mathbb{Z}$ with $i \neq j \mod n$, and are written
$(\!(i,j)\!).$  

We will be interested in certain factorizations into reflections of the element $\lambda_n \in \AS_n$ that acts on $\mathbb{R}^n$ by the translation \[\lambda_n: x \mapsto x+(1,1,\ldots,1,-n+1).\]
Then $\lambda_n$ can be expressed as a product of $2n-2$ reflections (and not fewer).   This $\lambda_n$ will play the role of the long cycle.

\subsection{Trees}
A reflection factorization  for $\lambda_n$
  \[\rr =  \big[(\!(a_{0}, b_1)\!), (\!(a_{1}, b_2)\!),\ldots, (\!(a_{2n-3}, b_{2n-2})\!)\big]\]
   is called \defn{tree-like} if $a_{k-1} < b_k$ and $a_k = b_k \mod n$.
  We write $\EFt_n$ for the set of all tree-like factorizations of $\lambda_n$.
  
{
\renewcommand{\thetheorem}{\ref{thm:tree_like}}
\begin{theorem}
  There is a bijection between $\EE_n$ and $\EFt_n$, where $\EE_n$ is the set of plane-embedded vertex-labeled trees on $[n]$ with a marked edge adjacent to the vertex $n$ (up to orientation preserving homeomorphism of the plane).%For any embedded vertex-labeled tree $T \in \EE_n$, $\rrT$ is a tree-like factorization.  For any tree-like factorization $\rr \in \EFt_n$, $\Trr$ is an embedded vertex-labeled tree.  The maps $T\mapsto \rrT$ and $\rr \mapsto \Trr$ are mutually inverse bijections between $\EE_n$ and $\EFt_n$.%\begin{align*}\EE_n &\simeq \EFt_n \\ T &\mapsto \rrT \\ \Trr &\mapsfrom \rr\end{align*} are mutually inverse bijections.
\end{theorem}
\addtocounter{theorem}{-1}
}

{
\renewcommand{\thecorollary}{\ref{cor:tree_like_enumeration}}
\begin{corollary}
For all $n\geq 2$, the number of tree-like factorizations of $\lambda_n$ is
\[\left|\EFt_n\right| = n! \Cat(n),\] where $\Cat(n) = \frac{1}{n+1}\binom{2n}{n}$ is the $n$th \defn{Catalan number}.
\end{corollary}
\addtocounter{theorem}{-1}
}

%\Nathan{Is it worth encoding this as a labeled Dyck path?  The point is that over steps and up steps that are paired get the same reflection, and you get to choose the labels arbitrarily (up to this rule).  When you do cyclic, you have a condition on what orderings are allowed.}

\subsection{Cyclic trees}

For each vertex-labeled tree, we will specify a preferred \emph{cyclic} embedding in the plane.  
Given a vertex-labeled tree $T \in \T_n$, its \defn{cyclic} embedding is given as follows: draw $T$ so that for every vertex $i \in [n]$ its neighboring vertices $j$ increase clockwise---with the exception that for $i\neq n$, $i$'s neighbor on the unique path from the vertex $n$ to $i$ is read as the central label $i$.  The marked edge is the edge from $n$ to its smallest neighbor.  An example of a cyclically-embedded tree is given in~\Cref{fig:main}.

The restriction of~\Cref{cor:tree_like_enumeration} to cyclically-embedded trees gives the notion of cyclic factorizations $\Ft_n$.

\begin{comment}
{
\renewcommand{\thetheorem}{\ref{thm:cyclic}}
\begin{theorem}
\label{thm:cyclic}
  For any cyclically-embedded vertex-labeled tree $T \in \E_n$, $\rrT$ is a cyclic factorization.  For any cyclic factorization $\rr\in\Ft_n$, $\Trr$ is a cyclically-embedded vertex-labeled tree.  The maps $T \mapsto \rrT$ and $\rr \mapsto \Trr$ are mutually inverse bijections between $\E_n$ and $\Ft_n$.% \begin{align*}\E_n &\simeq \Ft_n \\ T &\mapsto \rrT \\ \Trr &\mapsfrom \rr\end{align*} are mutually inverse bijections.
\end{theorem}
\addtocounter{theorem}{-1}
}
\end{comment}

\subsection{Subwords}
\label{sec:subwords_intro}  The factorizations in our version of the problem appear as labelings of certain Deodhar components for a braid variety constructed from an affine Kac-Moody group of type $A_{n-1}$.    Write $\S_n$ for the set of \defn{maximal distinguished subwords} of the word \[\blambda_n\coloneqq[s_0,s_1,\ldots,s_{n-1}]^{n-1}.\] For simplicity in the introduction, we define $\S_n$ to be the set of subwords with $2n-2$ skips whose product is the identity (the equivalence with the usual definition is proven in~\Cref{cor:distinguished}).  An example of an element of $\S_n$ is given in~\Cref{fig:main}.

%(this spells out a reduced word for $\lambda_n$)
\subsection{Subwords and cyclic trees}
Our main theorem is a bijection between $\S_n$ and (cyclically-embedded) trees.  

{
\renewcommand{\thetheorem}{\ref{thm:main_thm}}
\begin{theorem}
  There is a bijection between $\S_n$ and $\T_n$.
\end{theorem}
\addtocounter{theorem}{-1}
}

\subsection{Enumeration}
In previous work, $\S_n$ was counted by Galashin, Lam, and Williams using braid varieties, a trace formula in the affine Hecke algebra due to Opdam, and an identity due to Haglund~\cite{williamsopac}.

{
\renewcommand{\thetheorem}{\ref{thm:subword_count}}
\begin{theorem}[P.~Galashin, T.~Lam, N.~Williams]
\[
\left|R_{\blambda_n}(\FF_q) \right| =  (q - 1)^{2n-2}[n]_q^{n-2} \hspace{1em}\text{ and } \hspace{1em} |\S_n|=n^{n-2}.
\]
%\begin{align*}
%\left|R_{\blambda_n}(\FF_q) \right| &=  (q - 1)^{2n-2}[n]_q^{n-2} \text{ and} \\ |\S_n|&=n^{n-2}.
%\end{align*}  %Moreover,
%Let %$P_n = \mathcal{D}_{e,\w}$ be the set of subwords of $(s_0,s_1,\ldots,s_{n})^n$ of length $n(n-1)$ whose product is the identity and whose consecutive products decrease in weak order whenever possible.  
%\[|\S_n|=n^{n-2}.\]
\end{theorem}
\addtocounter{theorem}{-1}
}

\Cref{thm:main_thm,thm:subword_count} together give an elaborate new proof for the enumeration of $\T_n$.

{
\renewcommand{\thecorollary}{\ref{cor:cayley}}
\begin{corollary}[Cayley's formula]
$|\E_n|=n^{n-2}.$
\end{corollary}
\addtocounter{theorem}{-1}
}

\begin{remark}
Since the maximal distinguished subwords $\S_n$ are naturally in bijection with trees, it makes sense to consider the braid variety $R_{\blambda_n}(\FF_q)$ to be a $q$-analogue of vertex-labeled trees.  Compare with~\cite{leinster2021probability}, which gives a very different $q$-analogue as the number of nilpotent linear operators on $\FF_q^n$.
\end{remark}

%\subsection{Organization}
%\label{sec:organization}
The remainder of this paper has the same structure as the introduction, with a final section on future work.

%In~\Cref{sec:}, we recall background information on the affine symmetric group and introduce a certain factorization.  In~\Cref{sec:Trees}, we describe a \textit{cyclic} embedding of a tree on $n$ vertices and describe a map between these given trees and the factorizations described in~\cref{sec:}. We then go into the connection between these factorizations and subwords in the affine symmetric group in~\cref{sec:subwords} and finally describe the full bijection from $\S_n$ to $\textsc{tree}_n$ in~\cref{sec:bijection}. In~\Cref{sec:details}, we recall the affine Hecke algebra and Galashin, Lam, and Williams's enumeration of $\Ft_n$ using results of Opdam and Haglund.  Finally, in~\Cref{sec:future_work}, we state some open problems related to our work.
%\Nathan{Trees to cyclic factorizations; Subwords to cyclic factorizations.}

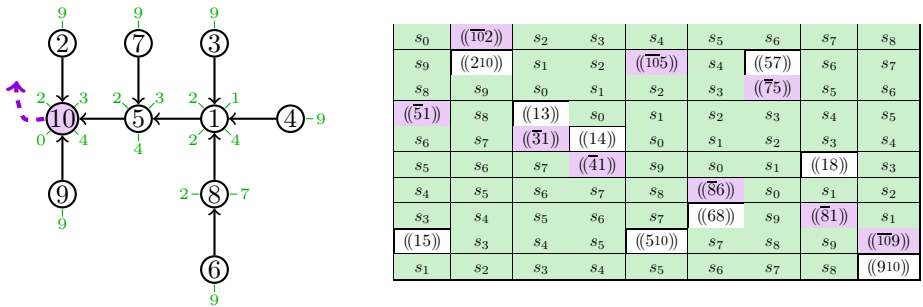
\begin{figure}[htbp]
\[
\raisebox{-.5\height}{
\begin{tikzpicture}
\node[draw,circle,inner sep=0pt,thick,minimum width=1em] (2) at (-1,1) {$2$};
\draw[-,darkblue] (-1,1.2) to (-1,1.3);
\node[darkblue] (l2) at (-1,1.4) {\tiny $9$};
\node[draw,circle,inner sep=0pt,thick,minimum width=1em] (9) at (-1,-1) {$9$};
\draw[-,darkblue] (-1,-1.2) to (-1,-1.3);
\node[darkblue] (l9) at (-1,-1.4) {\tiny $9$};
\node[draw,circle,inner sep=0pt,thick,minimum width=1em,black,fill=darkred!20,text opacity=1] (10) at (-1,0) {$10$};
\draw[-,darkblue] (-1+.16,.16) to (-1+.21,.21);
\node[darkblue] (l0) at (-1+.28,.28) {\tiny $3$};
\draw[-,darkblue] (-1+.16,-.16) to (-1+.21,-.21);
\node[darkblue] (l0) at (-1+.28,-.28) {\tiny $4$};
\draw[-,darkblue] (-1-.16,.16) to (-1-.21,.21);
\node[darkblue] (l0) at (-1-.28,.28) {\tiny $2$};
\draw[-,darkblue] (-1-.16,-.16) to (-1-.21,-.21);
\node[darkblue] (l0) at (-1-.28,-.28) {\tiny $0$};
\node[draw,circle,inner sep=0pt,thick,minimum width=1em] (5) at (0,0) {$5$};
\draw[-,darkblue] (0,-.2) to (0,-.3);
\node[darkblue] (l5) at (0,-.4) {\tiny $4$};
\draw[-,darkblue] (.14,.14) to (.21,.21);
\node[darkblue] (l5) at (.28,.28) {\tiny $3$};
\draw[-,darkblue] (-.14,.14) to (-.21,.21);
\node[darkblue] (l5) at (-.28,.28) {\tiny $2$};
\node[draw,circle,inner sep=0pt,thick,minimum width=1em] (7) at (0,1) {$7$};
\draw[-,darkblue] (0,1.2) to (0,1.3);
\node[darkblue] (l7) at (0,1.4) {\tiny $9$};
\node[draw,circle,inner sep=0pt,thick,minimum width=1em] (1) at (1,0) {$1$};
\draw[-,darkblue] (1.14,.14) to (1.21,.21);
\node[darkblue] (l5) at (1.28,.28) {\tiny $1$};
\draw[-,darkblue] (1.14,-.14) to (1.21,-.21);
\node[darkblue] (l5) at (1.28,-.28) {\tiny $4$};
\draw[-,darkblue] (.86,.14) to (.79,.21);
\node[darkblue] (l5) at (.72,.28) {\tiny $2$};
\draw[-,darkblue] (.86,-.14) to (.79,-.21);
\node[darkblue] (l5) at (.72,-.28) {\tiny $2$};
\node[draw,circle,inner sep=0pt,thick,minimum width=1em] (3) at (1,1) {$3$};
\draw[-,darkblue] (1,1.2) to (1,1.3);
\node[darkblue] (l3) at (1,1.4) {\tiny $9$};
\node[draw,circle,inner sep=0pt,thick,minimum width=1em] (8) at (1,-1) {$8$};
\draw[-,darkblue] (1.2,-1) to (1.3,-1);
\node[darkblue] (l8) at (1.4,-1) {\tiny $7$};
\draw[-,darkblue] (.8,-1) to (.7,-1);
\node[darkblue] (l8) at (.6,-1) {\tiny $2$};
\node[draw,circle,inner sep=0pt,thick,minimum width=1em] (6) at (1,-2) {$6$};
\draw[-,darkblue] (1,-2.2) to (1,-2.3);
\node[darkblue] (l6) at (1,-2.4) {\tiny $9$};
\node[draw,circle,inner sep=0pt,thick,minimum width=1em] (4) at (2,0) {$4$};
\draw[-,darkblue] (2.2,0) to (2.3,0);
\node[darkblue] (l4) at (2.4,0) {\tiny $9$};
\draw[->,thick] (2) -- (10);
\draw[->,thick] (9) -- (10);
\draw[->,thick] (5) -- (10);
\draw[->,thick] (7) -- (5);
\draw[->,thick] (1) -- (5);
\draw[->,thick] (3) -- (1);
\draw[->,thick] (4) -- (1);
\draw[->,thick] (8) -- (1);
\draw[->,thick] (6) -- (8);
\draw[->,dashed,ultra thick,darkred] plot [smooth] coordinates {(-1.24,0) (-1.5,0) (-1.6,.3) (-1.6,.5)};
\end{tikzpicture}} \hspace{2em} \scalebox{0.65}{
$\begin{array}{|c|c|c|c|c|c|c|c|c|} \hline
\cc s_0 & \cd (\!(\overline{\ten}2)\!) & \cc s_2 & \cc s_3 & \cc s_4 & \cc s_5 & \cc s_6 & \cc s_7 & \cc s_8 \\\hline
 \cc s_9 & (\!(2\ten)\!) & \cc s_1 & \cc s_2 & \cd  (\!(\overline{\ten}5)\!) & \cc s_4 & (\!(57)\!) & \cc s_6 & \cc s_7 \\\hline
 \cc s_8 & \cc s_9 & \cc s_0 & \cc s_1 & \cc s_2 & \cc s_3 & \cd  (\!(\overline{7}5)\!) & \cc s_5 & \cc s_6 \\\hline
 \cd (\!(\overline{5}1)\!) & \cc s_8 & (\!(13)\!) & \cc s_0 & \cc s_1 & \cc s_2 & \cc s_3 & \cc s_4 & \cc s_5 \\\hline
 \cc s_6 & \cc s_7 & \cd  (\!(\overline{3}1)\!) & (\!(14)\!) &
\cc s_0 & \cc s_1 & \cc s_2 & \cc s_3 & \cc s_4 \\\hline
 \cc s_5 & \cc s_6 & \cc s_7 & \cd (\!(\overline{4}1)\!) & \cc s_9 & \cc s_0 & \cc s_1 & (\!(18)\!) & \cc s_3 \\ \hline
 \cc s_4 & \cc s_5 & \cc s_6 & \cc s_7 & \cc s_8 &\cd  (\!(\overline{8}6)\!) & \cc s_0 & \cc s_1 & \cc s_2 \\\hline
 \cc s_3 & \cc s_4 & \cc s_5 & \cc s_6 & \cc s_7 & (\!(68)\!) & \cc s_9 &
\cd (\!(\overline{8}1)\!) & \cc s_1 \\\hline
 (\!(15)\!) & \cc s_3 & \cc s_4 & \cc s_5 & (\!(5\ten)\!) & \cc s_7 & \cc s_8& \cc s_9 &
\cd (\!(\overline{\ten}9)\!) \\\hline
 \cc s_1 & \cc s_2 & \cc s_3 & \cc s_4 & \cc s_5 & \cc s_6 & \cc s_7 & \cc s_8 & (\!(9\ten)\!) \\ \hline
\end{array}$}\]
\caption{Our running example.  {\it Left: } a cyclically-embedded vertex-labeled tree in $\T_{9}$ (for now, ignore the arrows, green edges, and green labels). \\ {\it Right:} the corresponding maximal distinguished subword $\u \in \S_{10}$, with takes in green, and skips in white and purple (decorated by the corresponding skip reflection, with the convention that $\overline{i}\coloneqq i-n$).}
\label{fig:main}
\end{figure}

\section{The affine symmetric group}
\label{sec:affine_symmetric}

The \defn{affine symmetric group} $\AS_{n}$ can be thought of as the group of bijections $\ww: \mathbb{Z} \to \mathbb{Z}$ such that~\cite[Chapter 8]{bjorner2005combinatorics}
\begin{equation}
\ww(i+n)=\ww(i)+n \text{ and } \sum_{i=1}^n \ww(i) = \binom{n+1}{2}.
\end{equation}

We write $(\!(i,j)\!)$ for the \defn{affine reflection} that interchanges $i$ and $j$ (simultaneously interchanging $i+kn$ and $j+kn$ for every $k \in \mathbb{Z}$); thus, $(\!(i,j)\!)=(\!(i+kn,j+kn)\!)$. We denote by $s_j \coloneqq (\!(j, j+1 )\!)$ the \defn{simple reflections}. For clarity of typesetting, we use the notation $\overline{i}\coloneqq i-n$.  We write $i \mod n$ for the representative between $1$ and $n$ equal to $i$ modulo $n$.

The \defn{reflection length} of $w \in \AS_n$ is the minimal number of reflections required to express $w$ as a product of reflections.

\begin{proposition}
\label{prop:descents}
For $0 \leq i < n$, $s_i=(\!(i,i+1)\!)$ is a right descent of $w\in \AS_n$ iff $w(i)>w(i+1)$.
\end{proposition}

\begin{proposition}
\label{eq:cycle}
  Let $\lambda_n$ be the translation that acts on $\mathbb{R}^n$ by \[\lambda_n: x \mapsto x+(1,1,\ldots,1,-n+1).\]Then $\lambda_n$ is an element of $\AS_n$ with:
  \begin{itemize}
    \item reduced word in simple reflections
      $(s_0s_1 \cdots s_{n-1})^{n-1}$;
    \item one-line notation
      \[\big[n+1,n+2,\ldots,2n-1,-n(n-2)\big]; \text{ and}\]
    \item cycle notation
      \begin{equation*}
        \left(\prod_{i=1}^{n-1}\left(\ldots,i-n,i,i+n,\ldots\right)\right)\left(\prod_{m=0}^{n-1} \left(\ldots,nm+n(n-1),nm,nm-n(n-1),\ldots\right)\right).
      \end{equation*}
  \end{itemize}
  Furthermore, $\lambda_n$ has reflection length $2n-2$.
\end{proposition}

\begin{proof}
The three descriptions are simple computations.  The reflection length is easily deduced from~\cite[Proposition 4.3]{mccammond2011bounding} (see also~\cite[Theorem 4.25]{lewis2019computing}).% the only partition of $(1,1,\ldots,1,-n)$ that sums to zero is the partition with a single block.
\end{proof}

From its cycle decomposition, we see that $\lambda_n$ acts on the integers as follows: it sends $k=0\mod n$ to $k-n(n-1)$, and it sends $k \neq 0 \mod n$ to $k + n$.

%%%%%%%%%%%%%%%%%%%%%%%%%%%%%%%%%%%%%%%%%%%%%%%%%%%%%%%%%%
\section{Trees}
\label{sec:factorizations}
%%%%%%%%%%%%%%%%%%%%%%%%%%%%%%%%%%%%%%%%%%%%%%%%%%%%%%%%%%

In~\Cref{sec:treelike_factorizations}, we describe sets of certain \emph{tree-like} factorizations of $\lambda_n$ in the affine symmetric group, which we will show in~\Cref{sec:trees} are encoded by clockwise walks around embedded vertex-labeled trees.

Given a finite sequence of reflections 
$\rr=\left[(\!(a_{0}, b_1)\!), (\!(a_{1}, b_2)\!),\ldots\right]$ and $k \in [n]$, write $\rrr_\ell=(\!(a_{\ell-1},b_\ell)\!)$ for the $\ell$th reflection in the sequence.  We say that $\rr$ is a \defn{factorization} of $w \in \AS_n$ if $w = \prod_{i} (\!(a_i,b_{i+1})\!)$; it is of \defn{minimal length} if the number of reflections is equal to the reflection length of $w$.  For $w \in \AS_n$, write $\Fact(w)$ for the set of its minimal length reflection factorizations.

%%%%%%%%%%%%%%%%%%%%%%%%%%%%%%%%%%%%%%%%%%%%%%%%%%%%%%%%%%
\subsection{Tree-like factorizations}
\label{sec:treelike_factorizations}
%%%%%%%%%%%%%%%%%%%%%%%%%%%%%%%%%%%%%%%%%%%%%%%%%%%%%%%%%%

\begin{definition}
\label{def:treelike_factorizations}
  A minimal length reflection factorization
  \[\rr =  \big[(\!(a_{0}, b_1)\!), (\!(a_{1}, b_2)\!),\ldots, (\!(a_{2n-3}, b_{2n-2})\!)\big] \in \Fact(\lambda_n)\]
  is \defn{tree-like} if $a_{k-1} < b_k$ and $a_k = b_k \mod n$.
  We write $\EFt_n$ for the set of all tree-like factorizations of $\lambda_n$.
\end{definition}

\begin{example}\label{ex1}
For $n=10$, the following factorization (see~\Cref{fig:main}) is tree-like:
\begin{align*}\rr=\Big[ &(\!(\overline{\ten} 2)\!),(\!(2\ten)\!),(\!(\overline{\ten} 5)\!),(\!(57)\!),(\!(\overline{7} 5)\!),(\!(\overline{5} 1)\!),(\!(13)\!),(\!(\overline{3} 1)\!),(\!(14)\!),\\ &(\!(\overline{4} 1)\!),(\!(18)\!),(\!(\overline{8} 6)\!),(\!(68)\!),(\!(\overline{8} 1)\!),(\!(15)\!),(\!(5\ten)\!),(\!(\overline{\ten} 9)\!),(\!(9\ten)\!)\Big].\end{align*}
\end{example}

We say that a reflection $\rrr_\ell$ \defn{increases} an integer $k$ if
    \[
        \rrr_\ell \rrr_{\ell + 1} \cdots \rrr_{2n-2}(k) > \rrr_{\ell + 1} \cdots \rrr_{2n-2}(k),
    \]
    and we say $\rrr_\ell$ \defn{decreases} $k$ if
    \[
        \rrr_\ell \rrr_{\ell + 1} \cdots \rrr_{2n-2}(k) < \rrr_{\ell + 1} \cdots \rrr_{2n-2}(k).
    \]
    Since each $\rrr_\ell$ is a reflection, there exist unique distinct integers $a , b \in [n]$ such that $\rrr_\ell$ increases $a$ and decreases $b$.  For \[\rr =  \big[(\!(a_{0}, b_1)\!), (\!(a_{1}, b_2)\!),\ldots, (\!(a_{2n-3}, b_{2n-2})\!)\big] \in \EFt_n,\] write
\begin{equation}
\label{eq:nb}
\rrk=[b_{i_1}\mod n,\ldots,b_{i_\ell}\mod n]
\end{equation}
for the sequence of $b_{i_j}$ (modulo $n$) for which $a_{i_j-1}=k \mod n$ (the abbreviation is for \emph{neighbors}).

\begin{example}
\label{ex:progression}
If we track the progression of the integer $0$ in \Cref{ex1} as the list of products $\rrr_\ell \rrr_{\ell+1} \cdots \rrr_{2n-2}(0)$ for $\ell=1,\ldots,2n-2$, we obtain the sequence
\begin{align*}\big[&-90,-88,-80,-75,-73,-65,-59,-57,-49,\\&-46,-39,-32,-24,-22,-19,-15,-10,-1\big].\end{align*}
Note that every reflection decreases $0$.   On the other hand, $1$ is unchanged until the $15$th reflection $(\!(15)\!)$, which increases $1$ to $5$, and is next modified by the $6$th reflection $(\!(\overline{5}1)\!)$, which increases $5$ to $11$.  Observe that exactly two reflections increase $1$.%obtain:
%\[[11,11,11,5,5,5,5,5,5,5,5,5,1,1,1].\]
\end{example}

We make the observations of~\Cref{ex:progression} precise in the following proposition, 
which gives a condition on factorizations equivalent to being tree-like.  This equivalent condition will be easier to check on the factorizations arising from distinguished subwords in~\Cref{sec:subwords}.   

\begin{proposition}
\label{connection}
    A factorization $\rr = \left[ \rrr_1,\rrr_2,\ldots,\rrr_{2n-2}\right] \in \Fact(\lambda_n)$ is tree-like if and only if there exist $a_0,\dots,a_{2n-2} \in \mathbb{Z}$ such that
\begin{equation}\label{eq:cond3}\rrr_\ell = (\!(a_{\ell - 1}, a_\ell)\!) \text{ and }
   |a_\ell - a_{\ell - 1}| < n \text{ for } \ell = 1,\dots,2n-2.\end{equation}
\end{proposition}
\begin{proof}
    First suppose that $\rr \in \Fact(\lambda_n)$ is tree-like.  Then we can choose
    \[
        a_0 < \cdots < a_{2n-2}
    \]
    such that $\rrr_\ell = (\!(a_{\ell -1}, a_\ell)\!)$.  Since for any $1 \leq \ell < 2n-2$
    \[
        \rrr_\ell \cdots \rrr_{2n-2}(a_{2n-2}) = a_{\ell-1} < a_{2n-2},
    \]
    it follows that every $\rrr_\ell$ decreases $a_{2n-2}$. So we must have $a_0 = a_{2n-2} = 0 \mod n$ because the only integers that $\lambda_n$ decreases are the multiples of $n$.

    Now since $\lambda_n$ maps each $k \in [n-1]$ to $k+n$, there must exist at least two $\rrr_\ell$ which either increase or decrease $k$. But each $\rrr_\ell$ decreases $n$ and increases some $k\neq n$, so by a pigeonhole argument, there are in fact exactly two unique factors $\rrr_\ell$ for each $k\neq n$ which increase $k$. If $\rrr_i$ and $\rrr_j$ are \emph{the} two factors which increase $k$, then
    \[
        k+n = \rrr_1\cdots \rrr_{2n-2}(k) = k + (a_i - a_{i-1}) + (a_j - a_{j-1}).
    \]
    So we must have $a_\ell - a_{\ell-1} < n$ for all $\ell$.
\medskip

     For the other direction, fix $\rr \in \Fact(\lambda_n)$ satisfying~\Cref{eq:cond3}.
          Note that the second condition $|a_\ell - a_{\ell-1}| < n$ implies that for a given $k \neq n$ there must exist $1 \leq i < j \leq 2n-2$ such that $\rrr_i$ and $\rrr_j$ increase $k$. Since we only have $2n-2$ factors, it follows again from a pigeonhole argument that each factor \emph{must} increase some $k \neq n$, and for each $k \neq n$, there are exactly two $\rrr_\ell$ which increase $k$.

    Now, since $0$ is sent to $-n(n-1)$ and each $\rrr_\ell$ can only decrease $0$ by at most $n-1$, we have that $0$ needs to be decreased by at least $n$ of the $\rrr_\ell$. Note that $a_{2n-2}$ is either increased or decreased by every $\rrr_\ell$. If $a_{2n-2} \neq 0 \mod n$, then $a_{2n-2}$ is increased by exactly two of the $\rrr_\ell$ and decreased by all the others. Since each $\rrr_\ell$ decreases only one integer $\mod n$, it follows that the two factors which increase $a_{2n-2}$ are the only factors that can decrease $n$. So we must have $n=2$. But in that case, there is only one minimal length reflection factorization of $\lambda_2$ satisfying~\Cref{eq:cond3}: \[(\!(0,1)\!)(\!(1,2)\!) = (\!(3,2)\!)(\!(2,1)\!).\] This factorization is tree-like, so the equivalence also holds for $n=2$.

    For $n > 2$, it follows that $a_{2n-2} = 0 \mod n$.  Since none of the $\rrr_\ell$ increase $0$ (since every $\rrr_\ell$ increases a $k \in [n-1]$), it follows that all $\rrr_\ell$ must decrease $0$. This implies that
    \[
        a_0 < a_1 < \cdots < a_{2n-2},
    \] 
    so that $\rr$ is tree-like.
\end{proof}

Since every reflection $r_\ell$ must decrease $0$, we immediately obtain the following corollaries.
\begin{corollary}
\label{cor:zero}
Let $\rr \in \EFt_n$ with $\rrr_\ell = (\!(a_{\ell-1}, a_\ell)\!)$.  Then $a_0 = a_{2n-2} = 0 \mod n$.
\end{corollary}

\begin{corollary}
\label{cor:pairs}
For $k \in [n-1]$, there are exactly two reflections that increase $k$---the first and last reflections to use a number equal to $k \mod n$.  We denote these reflections 
\begin{align*}
\rrs=(\!(\overline{b}_k,k)\!) &\text{ and } \rrf=(\!(k,b_k)\!) \text{ if } b_k>k \text{ and}\\
\rrs=(\!(b_k,k)\!) &\text{ and } \rrf=(\!(\overline{k},b_k)\!) \text{ if } b_k<k.
\end{align*}
We call $\rrs$ the \defn{left end} of its pair and $\rrf$ the \defn{right end} of its pair.
%are $\rr_i=(\!(\overline{b_k},k)\!)$ and $\rr_j=(\!(k,b_k)\!)$ if $b_k>k$, and $\rr_i=(\!(b_k,k)\!)$ and $\rr_j=(\!(\overline{k},b_k)\!)$ if $b_k<k$.
\end{corollary}

\begin{proof}
In the proof of \cref{connection}, we observed that for each $k\neq n$, there are unique factors $\rrr_{i}$ and $\rrr_{j}$, where $i < j$, which increase $k$, and every $\rrr_\ell$ is in one of these pairs. Assume $b_k>k$.  Then $\rrr_\ell\cdots \rrr_{2n-2}(k) = k$ for $\ell > j$, so $\rrr_j = (\!(k, b_k)\!)$ for some $b_k > k$. Moreover, 
    \[
        \rrr_\ell\cdots \rrr_{2n-2}(k) = \begin{cases}
            b_k & \text{if } i< \ell \leq j, \\
            k+n & \text{if } \ell \leq i,
        \end{cases}
    \]
    so $\rrr_i = (\!(b_k, k+n)\!) = (\!(\overline{b}_k,k)\!)$.  The case for $b_k<k$ is similar.
\end{proof}

\subsection{Tree embeddings}
\label{sec:trees}

\begin{definition}
We write $\EE_n$ for the set of plane-embedded vertex-labeled trees on $[n]$ with a marked edge adjacent to the vertex $n$ (up to orientation preserving homeomorphism of the plane).
\end{definition}
%%%%%%%%%%%%%%%%%%%%%%%%%%%%%%%%%%%%%%%%%%%%%%%%%%%%%%%%%%
%\subsection{All embeddings}
%\label{sec:treelike_from_trees}
%%%%%%%%%%%%%%%%%%%%%%%%%%%%%%%%%%%%%%%%%%%%%%%%%%%%%%%%%%

\begin{example}
The $30=3!\cdot \Cat(3)$ trees in $\EE_4$ are illustrated in~\Cref{fig:all_trees}.% (each tree corresponds to $3!$ elements by choosing a labeling of the unlabeled vertices by the numbers $1,2,3$).
\end{example}

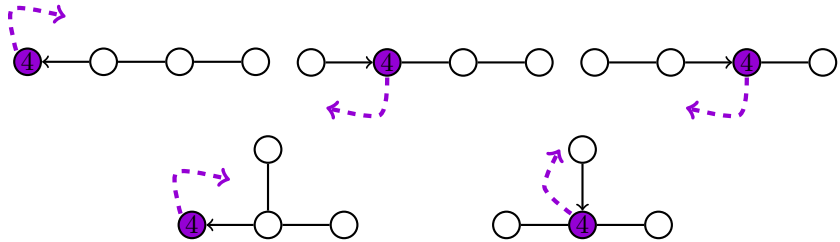
\begin{figure}[htbp]
\[
\begin{tikzpicture}
\node[draw,circle,inner sep=0pt,thick,minimum width=1em,black,fill=darkred,fill opacity=.2,text opacity=1] (1) at (0,0) {$4$};
\node[draw,circle,inner sep=0pt,thick,minimum width=1em] (2) at (1,0) {};
\node[draw,circle,inner sep=0pt,thick,minimum width=1em] (3) at (2,0) {};
\node[draw,circle,inner sep=0pt,thick,minimum width=1em] (4) at (3,0) {};
\draw[<-,thick] (1) -- (2);
\draw[-,thick] (2) -- (3);
\draw[-,thick] (3) -- (4);
%\node[draw,darkred,circle,fill=darkred,inner sep=0pt,minimum size=3pt] (s) at (-.2,.2) {};
\draw[->,dashed,ultra thick,darkred] plot [smooth] coordinates {(-.15,.15) (-.2,.7) (.5,.6)};
\end{tikzpicture} \hspace{1em}
\raisebox{-.62\height}{\begin{tikzpicture}
\node[draw,circle,inner sep=0pt,thick,minimum width=1em] (1) at (0,0) {};
\node[draw,circle,inner sep=0pt,thick,minimum width=1em,black,fill=darkred,fill opacity=.2,text opacity=1] (2) at (1,0) {$4$};
\node[draw,circle,inner sep=0pt,thick,minimum width=1em] (3) at (2,0) {};
\node[draw,circle,inner sep=0pt,thick,minimum width=1em] (4) at (3,0) {};
\draw[->,thick] (1) -- (2);
\draw[-,thick] (2) -- (3);
\draw[-,thick] (3) -- (4);
\draw[->,dashed,ultra thick,darkred] plot [smooth] coordinates {(1,-.2) (.9,-.7) (.2,-.6)};
\end{tikzpicture}} \hspace{1em}
\raisebox{-.62\height}{\begin{tikzpicture}
\node[draw,circle,inner sep=0pt,thick,minimum width=1em]  (1) at (0,0) {};
\node[draw,circle,inner sep=0pt,thick,minimum width=1em]  (2) at (1,0) {};
\node[draw,circle,inner sep=0pt,thick,minimum width=1em,black,fill=darkred,fill opacity=.2,text opacity=1] (3) at (2,0) {$4$};
\node[draw,circle,inner sep=0pt,thick,minimum width=1em]  (4) at (3,0) {};
\draw[-,thick] (1) -- (2);
\draw[->,thick] (2) -- (3);
\draw[-,thick] (3) -- (4);
\draw[->,dashed,ultra thick,darkred] plot [smooth] coordinates {(2,-.2) (1.9,-.7) (1.2,-.6)};
\end{tikzpicture}}\]
\[\begin{tikzpicture}
\node[draw,circle,inner sep=0pt,thick,minimum width=1em] (1) at (0,0) {};
\node[draw,circle,inner sep=0pt,thick,minimum width=1em] (2) at (0,1) {};
\node[draw,circle,inner sep=0pt,thick,minimum width=1em,black,fill=darkred,fill opacity=.2,text opacity=1] (3) at (-1,0) {$4$};
\node[draw,circle,inner sep=0pt,thick,minimum width=1em] (4) at (1,0) {};
\draw[-,thick] (1) -- (2);
\draw[->,thick] (1) -- (3);
\draw[-,thick] (1) -- (4);
\draw[->,dashed,ultra thick,darkred] plot [smooth] coordinates {(-1.15,.15) (-1.2,.7) (-.5,.6)};
\end{tikzpicture}
\hspace{5em}
 \begin{tikzpicture}
\node[draw,circle,inner sep=0pt,thick,minimum width=1em,black,fill=darkred,fill opacity=.2,text opacity=1] (1) at (0,0) {$4$};
\node[draw,circle,inner sep=0pt,thick,minimum width=1em] (2) at (0,1) {\phantom{$1$}};
\node[draw,circle,inner sep=0pt,thick,minimum width=1em] (3) at (-1,0) {\phantom{$3$}};
\node[draw,circle,inner sep=0pt,thick,minimum width=1em] (4) at (1,0) {\phantom{$2$}};
\draw[<-,thick] (1) -- (2);
\draw[-,thick] (1) -- (3);
\draw[-,thick] (1) -- (4);
\draw[->,dashed,ultra thick,darkred] plot [smooth] coordinates {(-.15,.15) (-.5,.5) (-.3,1)};
\end{tikzpicture}\]
\caption{The $30=3!\cdot \Cat(3)$ trees in $\EE_4$.  Each tree has only the vertex $4$ labeled, and so corresponds to $3!$ vertex-labeled trees in $\EE_4$ by choosing a labeling of the unlabeled vertices by $1,2,3$.}
\label{fig:all_trees}
\end{figure}

Given an embedded vertex-labeled tree $T\in \EE_n$, we produce a sequence $\rrT$ of $2n-2$ affine reflections (this sequence will turn out to be a tree-like factorization of $\lambda_n$).  Starting at the vertex labeled $n$, walk around the embedded tree $T$ clockwise (so that every edge is traversed exactly twice)---initially walking along the marked edge adjacent to $n$.  Record the edges visited as
  \begin{equation}
  \label{eq:vertices}
    \Big[(v_0,v_1),(v_1,v_2),\ldots,(v_{2n-3},v_{2n-2})\Big]
  \end{equation}
  with the convention that $v_0=n$ and $v_{2n-2}=n$. %Set $e_k$ to be edge $(v_k,v_{k+1})$. 
  Then $\rrT$ is defined to be the following sequence of $2n-2$ affine reflections:
  \begin{align*}
  \rrT&\coloneqq\Big[(\!(i_1,j_1)\!),(\!(i_2,j_2)\!),\ldots, (\!(i_{2(n-1)},j_{2(n-1)})\!)\Big], \text{ where}\\ 
  (\!(i_k,j_k)\!) &\coloneqq \begin{cases} (\!(v_{k-1},v_k)\!) & \text{if }v_{k-1}<v_k \\ (\!(\overline{v}_{k-1},v_k)\!) & \text{if } v_{k-1} > v_k  \end{cases} \text{ and } \overline{i}=i-n.\end{align*}

Conversely, given $\rr \in \EFt_n$, we define a corresponding plane embedded tree $\Trr$.  $\Trr$ has vertex set $[n]$; for each reflection $\rr_i$ that can be written as $(\!(a,b)\!)$ with $1 \leq a < b \leq n$, there is an edge connecting $a$ and $b$ (exactly half the reflections have this property by~\Cref{cor:pairs}).  It follows from \cref{def:cyclic_factorizations} that $\Trr$ is connected, and so it must be a tree since it has only $n-1$ edges.  The embedding of $\Trr$ is determined by placing the neighbors of vertex $k$ clockwise around $k$ in the order in which they appear in $\rrk$.

\begin{example}
  The factorization $\rr$ in~\Cref{ex1} corresponds to the embedded tree $T$ in~\Cref{fig:main}, where the marked edge adjacent to $10$ is $(10,2)$.
\end{example}

\begin{theorem}
\label{thm:tree_like}
  For any embedded vertex-labeled tree $T \in \EE_n$, $\rrT$ is a tree-like factorization.  For any tree-like factorization $\rr \in \EFt_n$, $\Trr$ is an embedded vertex-labeled tree.  The maps $T\mapsto \rrT$ and $\rr \mapsto \Trr$ are mutually inverse bijections between $\EE_n$ and $\EFt_n$.%\begin{align*}\EE_n &\simeq \EFt_n \\ T &\mapsto \rrT \\ \Trr &\mapsfrom \rr\end{align*} are mutually inverse bijections.
\end{theorem}

\begin{proof}
We first show that $\rrT$ is a factorization in $\Fact(\lambda_n)$. For $1 \leq i \leq n-1$ and $m \in \mathbb{Z}$, we wish to show that $i+mn$ is sent to $i+(m+1)n$ to conclude that the composition of reflections produced by $\rrT$ gives the first product of cycles in~\Cref{eq:cycle}.  By periodicity, it is enough to show this for $m=0$.

We will compute the composition of the reflections from right to left and show that we obtain $\lambda_n$.  We record the list of reflections by starting at the vertex labeled $n$ and walking around the tree $T$ counterclockwise, initially walking along the marked edge incident with $n$.  This allows us to read the list of vertices in~\Cref{eq:vertices} from right to left.  Suppose the first edge using the vertex $i$ encountered on this counterclockwise walk is the edge $(j,i)$ traversed from the vertex $j$ to $i$.

\begin{itemize}
\item If $i<j$, then we record the reflection $(\!(i,j)\!)$, which sends $i$ to $j$---and all edges encountered until we revisit the edge $(i,j)$ (now traversed from the vertex $i$ to $j$) do not involve $j$.  The second time the edge is revisited, we record the reflection $(\!(j-n,i)\!)=(\!(j,i+n)\!)$, and thus sends $j$ to $i+n$.
\item If $i>j$, then we record the reflection $(\!(i-n,j)\!)$, which sends $i$ to $j+n$.  Until we walk on this edge again, all other edges will not affect $j$. The second time the edge is revisited, we record the reflection $(\!(j,i)\!)=(\!(j+n,i+n)\!)$, which sends $j+n$ to $i+n$.
\end{itemize}
In each case, we conclude that $i$ is sent to $i+n$.

%that is, for $0 \leq k \leq n-2$ and $m \in \mathbb{Z}$, we wish to show that $nk+mn(n-1)$ is sent to $nk+(m-1)n(n-1)$.  B
It remains to show that $\rrT$ also gives the second product of cycles in~\Cref{eq:cycle}---again by periodicity, it is enough to show $n$ is sent to $n-n(n-1)$.  Since for every $1 \leq i<j \leq n$ every edge $(i,j)$ is traversed twice, once as just $(i,j)$ and once as $(j-n,i)$, and since every pair of adjacent reflections share a letter, $n$ is subtracted from the quantity exactly $(n-1)$ times---once for each pair of edges on the walk.  Thus, $n$ is sent to $n-n(n-1)$, as desired. 

By construction, the factorization $\rrT$ of $\lambda_n$ satisfies \cref{def:treelike_factorizations}, and so is tree-like.  It is clear that the inverse is given by the map $\rr \mapsto \Trr$.
\end{proof}

\begin{corollary}
\label{nest}
Let $\rr = \left[ \rrr_1,\rrr_2,\ldots,\rrr_{2n-2}\right] \in \EFt_n$, and write $\rrr_\ell = (\!(a_{\ell-1}, a_\ell)\!)$ with
$    a_0 < \cdots < a_{2n-2}$.  Then
\begin{enumerate}[label=(\roman*)]
        \item\label{nest2} If $(\!(a_{\ell-1}, a_\ell)\!)$ is to the left of $\rrs$ or to the right of $\rrf$, then $a_{\ell-1}, a_\ell \not= k \mod n$.
        
        \item\label{nest3} If $(\!(a_{\ell-1}, a_\ell)\!)$ is between $\rrs$ and $\rrf$, then $a_{\ell-1}, a_\ell \not= b_k \mod n$. 

        \item\label{nest4} $(\!(a_{\ell-1}, a_\ell)\!)$ is between $\rrs$ and $\rrf$ iff $(\!(\overline{a}_\ell,a_{\ell-1})\!)$ is also between $\rrs$ and $\rrf$.
    \end{enumerate}
\end{corollary}

\subsection{Enumeration}

The bijection of~\Cref{thm:tree_like} gives the following interesting enumeration for the tree-like factorizations of $\lambda_n$.

\begin{corollary}
\label{cor:tree_like_enumeration}
For all $n\geq 2$, the number of tree-like factorizations of $\lambda_n$ is
\[\left|\EFt_n\right| = n! \Cat(n),\] where $\Cat(n) = \frac{1}{n+1}\binom{2n}{n}$ is the $n$th \defn{Catalan number}.
\end{corollary}

\begin{proof}
The number of rooted planar trees with $n$ vertices is $\Cat(n)$.  By marking an edge, we remove any symmetries.  Since the number of plane-embedded vertex-labeled trees with $n$ vertices is $n!\Cat(n)$, we conclude the same enumeration for $\EFt_n$ by~\Cref{thm:tree_like}.
\end{proof}

\begin{remark}
\Cref{cor:tree_like_enumeration} is not our titular ``elaborate proof''---we are relying on previous combinatorial enumerations of rooted planar trees.  The issue is that we do not know what braid varieties over the loop group for $\mathrm{SL}_n$ correspond to tree-like factorizations; Minh-T\^am Trinh has constructed certain ``generalized Steinberg varieties'' using unipotent elements that give this enumeration---but using the simple Lie group and not its loop group.
\end{remark}

%%%%%%%%%%%%%%%%%%%%%%%%%%%%%%%%%%%%%%%%%%%%%%%%%%%%%%%%%%
\section{Cyclic trees}
 \label{sec:cyclic_trees}
%%%%%%%%%%%%%%%%%%%%%%%%%%%%%%%%%%%%%%%%%%%%%%%%%%%%%%%%%%

In~\Cref{sec:cyclic_factorizations}, we describe sets of certain \emph{cyclic} factorizations of $\lambda_n$ in the affine symmetric group, which we will show in~\Cref{sec:cyclic_from_trees} are encoded by clockwise walks around \emph{cyclically-embedded} vertex-labeled trees.

\subsection{Cyclic factorizations}
 \label{sec:cyclic_factorizations}
%%%%%%%%%%%%%%%%%%%%%%%%%%%%%%%%%%%%%%%%%%%%%%%%%%%%%%%%%%

\begin{definition}
\label{def:cyclic_factorizations}
A tree-like factorization
\[\rr =  \big[(\!(a_{0}, b_1)\!), (\!(a_{1}, b_2)\!),\ldots, (\!(a_{2n-3}, b_{2n-2})\!)\big] \in \EFt_n\]
is \defn{cyclic} if
\begin{enumerate}[label=(\roman*)]
    \item\label{CF3} if $\rrk[n]=[b_{i_1},\ldots,b_{i_\ell}]$ then \[b_{i_1} < \cdots < b_{i_\ell}; \text{ and}\]
    \item\label{CF2} for any $1\leq k< n$, if $\rrk=[b_{i_1},\ldots,b_{i_\ell}]$, then there exists some $1 \leq j \leq \ell$ for which \[b_{i_j} < b_{i_{j+1}} < \cdots < b_{i_{\ell-1}}< k < b_{i_1} < \cdots < b_{i_{j-1}}.\]
\end{enumerate}
We write $\Ft_n$ for the set of all cyclic factorizations of $\lambda_n$.
\end{definition}

Note that $b_{i_\ell}$ is replaced by $k$ in \crefitem{def:cyclic_factorizations}{CF2}.

\begin{example}
  The tree-like factorization in~\Cref{ex1} is also cyclic. \crefitem{def:cyclic_factorizations}{CF3} is satisfied bececause $\rrk[10]=[2,5,9]$ and   $2<5<9$.  As an example of \crefitem{def:cyclic_factorizations}{CF2}, $\rrk[1]=[3,4,8,5]$ satisfies $1<3<4<8$.
\end{example}

Our goal now is to give an equivalent characterization of cyclic factorizations, again to more easily connect with the factorizations arising from trees in~\Cref{sec:cyclic_from_trees}. We will require the following easy lemma concerning cyclic orderings.

\begin{lemma}
\label{clockIneq}
    Suppose that $a,v,b \in \mathbb{Z}$ such that $v-n < a < v < b < v + n$. Let $1 \leq \widetilde{a}, \widetilde{v}, \widetilde{b} \leq n$ be the corresponding values modulo $n$. Then $a + n < b$ if and only if $\widetilde{a} < \widetilde{b} < \widetilde{v}$, $\widetilde{v} < \widetilde{a} < \widetilde{b}$, or $\widetilde{b} < \widetilde{v} < \widetilde{a}$.%one of the following is true:
%    \begin{enumerate}
 %       \item $\tilde{a} < \tilde{b} < \tilde{v}$,
  %      \item $\tilde{v} < \tilde{a} < \tilde{b}$,
   %     \item $\tilde{b} < \tilde{v} < \tilde{a}$.
    %\end{enumerate}
\end{lemma}
\begin{proof}
    Both statements are obviously equivalent to
    \[
        v-n < a < b-n < v < a+n < b < v+n.\qedhere
    \]
\end{proof}

The following proposition will be used to connect cyclic factorizations with distinguished subwords in~\Cref{sec:bijection}.

\begin{proposition}
\label{incClock}
    Suppose that $\rr \in \EFt_n$. By~\Cref{cor:zero}, we can write
    \[
        (\rrr_1\cdots \rrr_{j-1})\rrr_j(\rrr_{j-1}\cdots \rrr_1) = (\!(0,m_j)\!),
    \]
    for each $j = 1,\dots,2n-2$.  Then $\rr$ is cyclic if and only if
    \[
        m_1 < \cdots < m_{2n-2}.
    \]
\end{proposition}
\begin{proof}
   Suppose that $\rr$ is a cyclic factorization, and fix $1 \leq j < 2n-2$. Write $\rrr_j = (\!(a , v)\!)$ and $\rrr_{j + 1} = (\!( v,b)\!)$ with $a < v < b$. By adding a multiple of $n$ if necessary, we can assume that $\rrr_1\cdots \rrr_{j-1}(a) = 0$. There are four cases to consider\footnote{These four cases correspond to the four cases in~\Cref{fig:runleaflabels} and in~\Cref{def:run_leaves}.}, all of which will be handled using \cref{nest}:
    \begin{enumerate}[label=(\alph*)]
        \item Suppose $\rrr_j$ and $\rrr_{j + 1}$ are both the left ends of their pairs. Then $m_j = v < b = m_{j + 1}$.
        \item Suppose $\rrr_{j}$ is the right end of its pair and $\rrr_{j + 1}$ is the left end of its pair. Then $m_j = a + n$ and $m_{\ell+1} = b$. By \cref{connection} we also know that $v-n < a < v < b < v + n$. It then follows from \crefitem{def:cyclic_factorizations}{CF2} and \cref{clockIneq} that $m_j < m_{j + 1}$.
        \item Suppose $\rrr_j$ and $\rrr_{j + 1}$ are both the right ends of their pairs. Then $m_j = a + n < v + n = m_{j + 1}$.
        \item Suppose $\rrr_{j}$ is the left end of its pair and $\rrr_{j + 1}$ is the right end of its pair. Then these must be the same pair, so $a = b-n$, and $m_j = v < v + n = m_{j + 1}$.
    \end{enumerate}
\medskip

    Suppose now that we have a tree-like factorization with $m_1 < \cdots < m_{2n-2}$. We begin by considering \crefitem{def:cyclic_factorizations}{CF3}.  It follows from \cref{nest} that we can write $\rrr_1,\dots,\rrr_{2n-2}$ as
    \[
        (\!(\bar{n}, v_1)\!),\dots,(\!(v_1,n)\!),(\!(\bar{n},v_2)\!),\dots,(\!(v_2,n)\!),\dots,(\!(\bar{n},v_\ell)\!),\dots,(\!(v_\ell,n)\!).
    \]
    Consider the adjacent factors $\rrr_{j} = (\!(v_i, n)\!)$ and $\rrr_{j+1} = (\!(\bar{n},v_{i+1})\!)$. We have $v_{i+1} - v_i = m_{j+1} - m_j > 0$, so $v_i < v_{i+1}$.

    Now consider \crefitem{def:cyclic_factorizations}{CF2}. It follows from \cref{nest} that we can write $\rrr_1,\dots,\rrr_{2n-2}$ as
    \[
        \dots, (\!(\bar{a}_\ell, k)\!), (\!(k,a_1)\!),\dots,(\!(\bar{a}_1,k)\!),\dots,(\!(k,a_{\ell-1})\!), \dots,(\!(\bar{a}_{\ell-1},k)\!),(\!(k,a_\ell)\!),\dots,
    \]
    where each $a_i = v_i \mod n$. Consider the adjacent factors $\rrr_{j} = (\!(\bar{a}_i, k)\!)$ and $\rrr_{j+1} = (\!(k,a_{i+1})\!)$ for $1\leq i \leq \ell - 2$. We have $a_{i+1} - a_i = m_{j+1} - m_j > 0$, so $a_i < a_{i+1}$. We also have $k-n < a_{i} - n < k < a_{i+1} < k + n$, so by \cref{clockIneq} either $v_i < v_{i+1} < k$, $k < v_i < v_{i+1}$, or $v_{i+1} < k < v_i$.  \crefitem{def:cyclic_factorizations}{CF2} follows, so that $\rr$ is cyclic.
\end{proof}

%%%%%%%%%%%%%%%%%%%%%%%%%%%%%%%%%%%%%%%%%%%%%%%%%%%%%%%%%%
\subsection{Cyclic embeddings}
\label{sec:cyclic_from_trees}
%%%%%%%%%%%%%%%%%%%%%%%%%%%%%%%%%%%%%%%%%%%%%%%%%%%%%%%%%%

\begin{definition}
We write $\E_n$ for the set of vertex-labeled trees (as abstract graphs).
\end{definition}
 For each vertex-labeled tree, we will now specify a preferred \emph{cyclic} embedding in the plane.  (We note that there is some similarity with~\cite[Section 3]{goulden2002tree}.)

\begin{definition}
\label{def:cyclic_trees}
Given a vertex-labeled tree $T \in \T_n$, its \defn{cyclic} embedding is given as follows: draw $T$ so that for every vertex $i \in [n]$ its neighboring vertices $j$ increase clockwise---with the exception that for $i\neq n$, $i$'s neighbor on the unique path from the vertex $n$ to $i$ is read as the central label $i$.  The marked edge is the edge from $n$ to its smallest neighbor.
\end{definition}

To make the clockwise increasing condition easy to see in examples, we direct each edge in $T$ towards the vertex $n$.

\begin{example}
All 16 trees in $\T_4$ are drawn in~\Cref{fig:trees} in their cyclic embedding; a larger example is given in~\Cref{fig:main}.
\end{example}

\input{tree_figure.tex}
%fig:trees

By construction, \Cref{thm:tree_like} restricts from all tree-like factorizations and all embeddings to cyclic factorizations and embeddings.

\begin{theorem}
\label{thm:cyclic}
  For any cyclically-embedded vertex-labeled tree $T \in \E_n$, $\rrT$ is a cyclic factorization.  For any cyclic factorization $\rr\in\Ft_n$, $\Trr$ is a cyclically-embedded vertex-labeled tree.  The maps $T \mapsto \rrT$ and $\rr \mapsto \Trr$ are mutually inverse bijections between $\E_n$ and $\Ft_n$.% \begin{align*}\E_n &\simeq \Ft_n \\ T &\mapsto \rrT \\ \Trr &\mapsfrom \rr\end{align*} are mutually inverse bijections.
\end{theorem}

\begin{remark}
At this point we could use the known enumeration of $\E_n$ to conclude that $|\Ft_n|=n^{n-2}$.  We will instead connect $\Ft_n$ to certain maximal distinguished subwords in~\Cref{sec:subwords}, connect these subwords to certain braid varieties in~\Cref{sec:braid_varieties}, use representation-theoretic methods to compute the point count of the braid varieties over a finite field with $q$ elements, and then recover the cardinality of $\Ft_n$ by sending $q \to 1$.
\end{remark}

\begin{remark}
In analogy with the usual problem of minimal reflection factorizations of the long cycle in $S_n$ and the noncrossing partition lattice, it seems natural to define a partial order on the prefixes of cyclic factorizations in $\AS_n$.  Unfortunately, for $n\geq 4$ there are maximal chains in this partial order that no longer correspond to cyclic factorizations.
\end{remark}

\begin{comment}
\begin{proof}
If $T$ is cyclic, then so is $\rrT$.% by \cref{nest}.

If $\rr$ is cyclic, then the embedding of $\Trr$ induced by $\rr$ is described as follows:
    \begin{itemize}
        \item the neighbors of $n$ increase clockwise.
        \item for a vertex $k \neq n$, its neighbors increase clockwise when the neighbor $b_k$ is temporarily given the label $k$, where $b_k$ is as in \Cref{cor:pairs}.
    \end{itemize}
Therefore, $\Trr$ is cyclic.
\end{proof}
\end{comment}

%%%%%%%%%%%%%%%%%%%%%%%%%%%%%%%%%%%%%%%%%%%%%%%%%%%%%%%%%%
\section{Subwords}
\label{sec:subwords}
%%%%%%%%%%%%%%%%%%%%%%%%%%%%%%%%%%%%%%%%%%%%%%%%%%%%%%%%%%

A \defn{subword} $\u$ of a sequence $[s_{i_1},s_{i_2},\ldots, s_{i_m}]$ of simple generators of the affine symmetric group $\AS_n$ (see~\Cref{sec:affine_symmetric} for more details) is a sequence \[\mathsf{u}=[u_1,u_2,\dots,u_m], \text{ where } u_j\in\{s_{i_j},e\} \text{ for all } j.\]  We call the letters $j$ for which $u_j=e$, \defn{skips}, and the letters $j$ for which $u_j = s_{i_j}$ \defn{takes}.  For any such sequence, we set
\begin{align}\label{eq:partial_prod}
u_{(j)} &\coloneqq u_1u_2\cdots u_j \in \AS_n, \text{ and}\\
u^{(j)} &\coloneqq u_j\cdots u_{m} \in \AS_n.\nonumber
\end{align}

We say $\u$ is a \defn{$w$-subword} if $u_{(m)}=w$. 

\subsection{Maximal distinguished subwords}

\begin{definition}
\label{def:subwords}
Write $\S_n$ for the set of \defn{maximal distinguished subwords} of the word \[\blambda_n\coloneqq[s_0,s_1,\ldots,s_{n-1}]^{n-1}.\] That is, $\S_n$ is the set of subwords with $2n-2$ skips whose product is the identity.
\end{definition}
  The $n$ consecutive factors of length $n-1$ of the word $\blambda_n$---from the $i(n-1)$st letter to the $((i+1)(n-1)-1)$st letter---will be called \defn{rows}.  Drawing $\blambda_n$ with subsequence rows vertically aligned gives the notion of \defn{columns}.  We will typically depicting $\blambda_n$ or a subword $\u \in \S_n$ using an $n \times (n-1)$ array.

We will show in~\Cref{cor:distinguished} that for this special case of $\blambda_n$,~\Cref{def:subwords} recovers the usual notion of \emph{distinguished}~\cite{deodhar1985some}.

\begin{example}
The $16$ maximal distinguished subwords in $\S_4$ are given in~\Cref{fig:distinguished}.  A larger example is given in~\Cref{fig:main}.  See also~\Cref{fig:invs}.
\end{example}

In preparation to connect subwords to trees, we associate a reflection to each skip in a subword in $\S_n$.

\begin{definition}\label{def:inv}
For $\u \in \S_n$, define \begin{equation}\label{eq:inv}\Inv(\u) \coloneqq \big[r_1,r_2,\ldots,r_{n(n-1)}\big]\end{equation}
where $r_k=u_{(i_{k-1})}s_{i_k}u_{(i_{k-1})}^{-1}$ (the notation $u_{(i)}$ is defined in~\Cref{eq:partial_prod}).  We write $\rru$ for the subsequence of $\Inv(\u)$ obtained by restricting to the \emph{skips} of $\u$---that is, restricted to the indices $j$ for which $u_j = e$---and call the subsequence \defn{skip reflections}.
\end{definition}

\begin{remark}
We will show in~\Cref{sec:bijection} that $\S_n$ is in bijection with $\E_n$---the skip reflections will determine the edges of the corresponding tree.
\end{remark}

\begin{example}
\Cref{fig:invs} illustrates $\Inv(\u)$ for the maximal distinguished subword from~\Cref{fig:main}.
\end{example}

\begin{figure}[htbp]
\[\begin{array}{|c|c|c|c|c|c|c|c|c|} \hline (\!(\overline{\ten}1)\!)& \cd (\!(\overline{\ten}2)\!)&  (\!(23)\!)&  (\!(24)\!)&  (\!(25)\!)&  (\!(26)\!)&  (\!(27)\!)&  (\!(28)\!)&  (\!(29)\!)
\\ \hline (\!(\overline{2}1)\!)& \cd (\!(2\ten)\!)&  (\!(\overline{\ten}3)\!)&  (\!(\overline{\ten}4)\!)& \cd (\!(\overline{\ten}5)\!)&  (\!(56)\!)& \cd (\!(57)\!)&  (\!(78)\!)&  (\!(79)\!)
\\ \hline (\!(\overline{7}1)\!)&  (\!(27)\!)&  (\!(\overline{7}3)\!)&  (\!(\overline{7}4)\!)&  (\!(7\ten)\!)&  (\!(\overline{7}6)\!)& \cd (\!(\overline{7}5)\!)&  (\!(58)\!)&  (\!(59)\!)
\\ \hline\cd (\!(\overline{5}1)\!)&  (\!(\overline{2}1)\!)& \cd (\!(13)\!)&  (\!(34)\!)&  (\!(\overline{\ten}3)\!)&  (\!(36)\!)&  (\!(\overline{7}3)\!)&  (\!(38)\!)&  (\!(39)\!)
\\ \hline (\!(35)\!)&  (\!(23)\!)& \cd (\!(\overline{3}1)\!)& \cd (\!(14)\!)&  (\!(\overline{\ten}4)\!)&  (\!(46)\!)&  (\!(\overline{7}4)\!)&  (\!(48)\!)&  (\!(49)\!)
\\ \hline (\!(45)\!)&  (\!(24)\!)&  (\!(34)\!)& \cd (\!(\overline{4}1)\!)&  (\!(\overline{\ten}1)\!)&  (\!(16)\!)&  (\!(\overline{7}1)\!)& \cd (\!(18)\!)&  (\!(89)\!)
\\ \hline (\!(58)\!)&  (\!(28)\!)&  (\!(38)\!)&  (\!(48)\!)&  (\!(8\ten)\!)& \cd (\!(\overline{8}6)\!)&  (\!(\overline{7}6)\!)&  (\!(16)\!)&  (\!(69)\!)
\\ \hline (\!(56)\!)&  (\!(26)\!)&  (\!(36)\!)&  (\!(46)\!)&  (\!(6\ten)\!)& \cd (\!(68)\!)&  (\!(78)\!)& \cd (\!(\overline{8}1)\!)&  (\!(19)\!)
\\ \hline\cd (\!(15)\!)&  (\!(25)\!)&  (\!(35)\!)&  (\!(45)\!)& \cd (\!(5\ten)\!)&  (\!(6\ten)\!)&  (\!(7\ten)\!)&  (\!(8\ten)\!)& \cd (\!(\overline{\ten}9)\!)
\\ \hline (\!(19)\!)&  (\!(29)\!)&  (\!(39)\!)&  (\!(49)\!)&  (\!(59)\!)&  (\!(69)\!)&  (\!(79)\!)&  (\!(89)\!)& \cd (\!(9\ten)\!)\\ \hline \end{array}\]
\caption{$\Inv(\u)$ for the maximal distinguished subword from~\Cref{fig:main}.  Skips are colored purple.}
\label{fig:invs}
\end{figure}

\newcolumntype{C}[1]{>{\centering\arraybackslash$}p{#1}<{$}}
\begin{figure}[htbp]
\scalebox{0.8}{
$\begin{array}{ccc}
\begin{array}{|C{2em}|C{2em}|C{2em}|}\hline \cd (\!(\overline{4}1)\!) & \cc s_1 & (\!(13)\!) \\\hline \cc s_3 & \cd (\!(\overline{3}2)\!) & \cc s_1 \\\hline \cc s_2 & (\!(23)\!) & \cd (\!(\overline{3}1)\!) \\\hline (\!(14)\!) & \cc s_2 & \cc s_3 \\\hline \end{array} &
\begin{array}{|C{2em}|C{2em}|C{2em}|}\hline \cc s_0 & \cc s_1 & \cd (\!(\overline{4}3)\!) \\\hline \cc s_3 & \cd (\!(\overline{3}2)\!) & \cc s_1 \\\hline \cd (\!(\overline{2}1)\!) & \cc s_3 & \cc s_0 \\\hline (\!(12)\!) & (\!(23)\!) & (\!(34)\!) \\\hline \end{array} &
\begin{array}{|C{2em}|C{2em}|C{2em}|}\hline \cc s_0 & \cd (\!(\overline{4}2)\!) & (\!(23)\!) \\\hline \cd (\!(\overline{3}1)\!) & \cc s_0 & \cc s_1 \\\hline (\!(13)\!) & \cc s_3 & \cd (\!(\overline{3}2)\!) \\\hline \cc s_1 & (\!(24)\!) & \cc s_3
\\\hline \end{array}
  \\[3em]
\begin{array}{|C{2em}|C{2em}|C{2em}|}\hline \cc s_0 & \cc s_1 & \cd (\!(\overline{4}3)\!) \\\hline \cd (\!(\overline{3}1)\!) & (\!(12)\!) & \cc s_1 \\\hline \cc s_2 & \cd (\!(\overline{2}1)\!) & \cc s_0 \\\hline (\!(13)\!) & \cc s_2 & (\!(34)\!)  \\\hline \end{array} &
\begin{array}{|C{2em}|C{2em}|C{2em}|}\hline \cc s_0 & \cd (\!(\overline{4}2)\!) & \cc s_2 \\\hline \cd (\!(\overline{2}1)\!) & \cc s_0 & (\!(13)\!) \\\hline \cc s_2 & \cc s_3 & \cd (\!(\overline{3}1)\!) \\\hline (\!(12)\!) & (\!(24)\!) & \cc s_3 \\\hline \end{array} &
\begin{array}{|C{2em}|C{2em}|C{2em}|}\hline \cd (\!(\overline{4}1)\!) & (\!(12)\!) & (\!(23)\!) \\\hline \cc s_3 & \cc s_0 & \cd (\!(\overline{3}2)\!) \\\hline \cc s_2 & \cd (\!(\overline{2}1)\!) & \cc s_0 \\\hline (\!(14)\!) & \cc s_2 & \cc s_3 \\\hline \end{array}
\\[3em] \hline \\[-.5em]
\begin{array}{|C{2em}|C{2em}|C{2em}|}\hline \cd (\!(\overline{4}1)\!) & \cc s_1 & \cc s_2 \\\hline (\!(14)\!) & \cd (\!(\overline{4}2)\!) & (\!(23)\!) \\\hline \cc s_2 & \cc s_3 & \cd (\!(\overline{3}2)\!) \\\hline \cc s_1 & (\!(24)\!) & \cc s_3 \\\hline \end{array} &
\begin{array}{|C{2em}|C{2em}|C{2em}|}\hline \cc s_0 & \cd (\!(\overline{4}2)\!) & \cc s_2 \\\hline \cc s_3 & (\!(24)\!) & \cd (\!(\overline{4}3)\!) \\\hline \cd (\!(\overline{3}1)\!) & \cc s_3 & \cc s_0 \\\hline (\!(13)\!) & \cc s_2 & (\!(34)\!) \\\hline \end{array} &
\begin{array}{|C{2em}|C{2em}|C{2em}|}\hline \cd (\!(\overline{4}1)\!) & \cc s_1 & \cc s_2 \\\hline (\!(14)\!) & \cc s_0 & \cd (\!(\overline{4}3)\!) \\\hline \cc s_2 & \cd (\!(\overline{3}2)\!) & \cc s_0 \\\hline \cc s_1 & (\!(23)\!) & (\!(34)\!) \\\hline \end{array}
\\[3em] \hline \\[-.5em]
\begin{array}{|C{2em}|C{2em}|C{2em}|}\hline \cd (\!(\overline{4}1)\!)& (\!(12)\!) & \cc s_2 \\\hline \cc s_3 & \cd (\!(\overline{2}1)\!) & \cc s_1 \\\hline (\!(14)\!) & \cc s_3 & \cd (\!(\overline{4}3)\!) \\\hline \cc s_1 & \cc s_2 & (\!(34)\!) \\\hline \end{array} &
\begin{array}{|C{2em}|C{2em}|C{2em}|}\hline \cd (\!(\overline{4}1)\!) & \cc s_1 & (\!(13)\!) \\\hline \cc s_3 & \cc s_0 & \cd (\!(\overline{3}1)\!) \\\hline (\!(14)\!) & \cd (\!(\overline{4}2)\!) & \cc s_0 \\\hline \cc s_1 & (\!(24)\!) & \cc s_3 \\\hline \end{array} &
\begin{array}{|C{2em}|C{2em}|C{2em}|}\hline \cc s_0 & \cd (\!(\overline{4}2)\!) & \cc s_2 \\\hline \cd (\!(\overline{2}1)\!) & \cc s_0 & \cc s_1 \\\hline (\!(12)\!) & (\!(24)\!) & \cd (\!(\overline{4}3)\!) \\\hline \cc s_1 & \cc s_2 & (\!(34)\!) \\\hline \end{array}
\\[3em] \hline \\[-.5em]
\begin{array}{|C{2em}|C{2em}|C{2em}|}\hline \cd (\!(\overline{4}1)\!) & (\!(12)\!) & \cc s_2 \\\hline \cc s_3 & \cd (\!(\overline{2}1)\!) & (\!(13)\!) \\\hline  \cc s_2 & \cc s_3 & \cd (\!(\overline{3}1)\!) \\\hline (\!(14)\!) & \cc s_2 & \cc s_3 \\\hline \end{array}&
\begin{array}{|C{2em}|C{2em}|C{2em}|}\hline \cc s_0 & \cd (\!(\overline{4}2)\!) & (\!(23)\!) \\\hline \cc s_3 & \cc s_0 & \cd (\!(\overline{3}2)\!) \\\hline \cd (\!(\overline{2}1)\!) & \cc s_3 & \cc s_0 \\\hline (\!(12)\!) & (\!(24)\!) & \cc s_3 \\\hline \end{array} &
\begin{array}{|C{2em}|C{2em}|C{2em}|}\hline \cc s_0 & \cc s_1 & \cd (\!(\overline{4}3)\!) \\\hline \cd (\!(\overline{3}1)\!)& \cc s_0 & \cc s_1 \\\hline (\!(13)\!) & \cd (\!(3\overline{2})\!) & \cc s_0 \\\hline \cc s_1 & (\!(23)\!) & (\!(34)\!) \\\hline \end{array}
\\[3em] \hline \\[-.5em]
& \begin{array}{|C{2em}|C{2em}|C{2em}|}\hline \cd (\!(\overline{4}1)\!) & \cc s_1 & \cc s_2 \\\hline (\!(14)\!)& \cd (\!(\overline{4}2)\!) & \cc s_1 \\\hline \cc s_2 & (\!(24)\!) & \cd (\!(\overline{4}3)\!) \\\hline \cc s_1 & \cc s_2 & (\!(34)\!) \\\hline \end{array} &\end{array}$}
\caption{The $16$ distinguished subwords in $\S_4$, with letters chosen in the subword indicated in green, positive skips in white, and negative skips in purple (and replaced by the corresponding inversions).  Compare with~\Cref{fig:trees}.}
\label{fig:distinguished}
\end{figure}

For a subword $\u = [u_1,\dots, u_{n(n-1)}]$ of $\blambda_n$, we encode the pattern of skips in an \defn{indicator word} $\psi(\u) = [\psi_1, \dots, \psi_{n(n-1)}]$, where
\[
    \psi_i = \begin{cases}
        0 & \text{if } u_i = e, \\
        1 & \text{otherwise}.
    \end{cases}
\]

We define the \defn{rotation} of a subword, denoted $\operatorname{rot}(\mathsf{u})$ as the subword with indicator word 
\[
    \left[\psi(\u)_{n(n-1)}, \psi(\u)_1, \dots, \psi(\u)_{n(n-1) - 1}\right].
\]
Explicitly, we can write $\operatorname{rot}(\u)_i$ in terms of how it acts on the integers via
\[
    \operatorname{rot}(\u)_i = \alpha_+ u_{i - 1} \alpha_-,
\]
where $\alpha_+(m)= m+1$ and $\alpha_- (m)= m-1$. %Rotation will be useful in reducing some proofs to checking a single case.

\begin{lemma}
\label{rotOfEword}
    The rotation of an $e$-subword $\u$ of $\blambda_n$ is an $e$-subword.
\end{lemma}
\begin{proof}
    Write $u'_i \coloneqq \operatorname{rot}(\u)_i$. Then
    \begin{align*}
        u'_1\cdots u'_{n(n-1)} &= \alpha_+ u_{n(n-1)}u_1\cdots u_{n(n-1) - 1}\alpha_- \\
        &= \alpha_+ u_{n(n-1)}(u_1\cdots u_{n(n-1)})u_{n(n-1)}\alpha_- = e.\qedhere
    \end{align*}
\end{proof}

\begin{corollary}
    The rotation of a maximal distinguished subword of $\blambda_n$ is again a maximal distinguished subword.
\end{corollary}
%\begin{proof}
%    This follows immediately from \cref{rotOfEword} because the number of skips is maintained by rotation.
%\end{proof}
\begin{remark}
It might be interesting to determine the orbit structure of $\S_n$ under cyclic rotation.
\end{remark}

\subsection{Skip reflections are tree-like}
Our eventual goal is to show that if $\u \in \S_n$, then $\rru$ gives a cyclic factorization of $\lambda_n$.  We begin by showing that $\rru$ is tree-like (\Cref{def:treelike_factorizations}).

 \begin{proposition}
\label{invProd}
    For $\u$ a subword of $\blambda_n$, write $\rru = \left[\rrr_1,\dots,\rrr_{k}\right]$.  Let $i_j$ denote the index of the skip corresponding to $\rrr_j$. Then for any $i_j \leq \ell < i_{j+1}$,
    \[
        \rrr_1 \cdots \rrr_{j} = (\blambda_n)_{(\ell)} u_{(\ell)}^{-1}.
    \]
\end{proposition}
\begin{proof}
    This is~\cite[Proposition 4.7]{galashin2022rational}.  The idea is to notice that \[(\blambda_n)_{(\ell)}^{-1}\rrr_1\cdots \rrr_{j} = \u_{(\ell)}^{-1}\] because the $\rrr_i$ cancel the corresponding skips from $(\blambda_n)_{(\ell)}^{-1}$.
\end{proof}

\begin{corollary}
    If $\u \in \S_n$, then $\rru \in \Fact(\lambda_n)$.
\end{corollary}

We can understand how a subword $\u$ acts on $k$ by looking at $\Inv(\u)$.  The $j$th reflection $(\!(a,b)\!)$ in $\Inv(\u)$ with $a$ or $b$ equal to $k$ modulo $n$ will be either:
\begin{itemize}
  \item \defn{increases} ($u^{(j)}(k) > u^{(j+1)}(k)$), which occur along columns
  \item \defn{decreases} ($u^{(j)}(k) < u^{(j+1)}(k)$), which occur along rows
  \item skips ($u^{(j)}(k) = u^{(j+1)}(k)$), which appear as corners.
\end{itemize}

%Increases in this array occur along columns (in particular, at most one per row except), while decreases occur along rows.
%It is helpful to record how a subword $\u$ acts on $k$ as the array
%\[
%\left[\begin{array}{ccc}
%u_{1,1}(k)& \ldots & u_{1,n-1}(k)\\
%\vdots & \ddots & \vdots \\
%u_{n,1}(k)& \ldots& u_{n,n-1}(k)
%\end{array}\right],\] where $u_{i,j}$ is the product $u_{(n-1)(i-1)+j}\cdots u_{n(n-1)}.$  

\begin{example}
It is helpful for understanding the subsequent proofs to interpret the various data on subwords in terms of cyclic trees (although the bijection will not be formally proven until~\Cref{sec:bijection}).

Consider the set of reflections $(\!(a,b)\!)$ in $\Inv(\u)$ with $a$ or $b$ equal to $k$ modulo $n$.  This set, when highlighted on $\blambda_n$ drawn using $n$ rows and $n-1$ columns, record what is seen as one goes clockwise around the vertex $k$ in the tree corresponding to $\rru$.  The set forms a connected path heading down and right on a torus (if the path goes below the $n$th row, the path continues in the same column in the first row; if the path goes to the right of the $(n-1)$st column, the path continues one row below in the first column):
\begin{itemize}
\item for $k\neq 0 \mod n$, the path starts down from the topmost box in column $k$; if $k=0 \mod n$, then then path starts right from the top left box;
\item skips of $\u$ correspond to neighbors of $k$; at skips, the path switches between going down and going right;
\item neighbors $a<k$ (resp. $a>k$) are recorded as corners $(\!(a,k)\!)$ and $(\!(\overline{k},a)\!)$ (resp. corners $(\!(\overline{a},k)\!)$ and $(\!(k,a)\!)$) in the same column of the strip;
\item the vertical distance between the corners $(\!(a,k)\!)$ and $(\!(\overline{k},a)\!)$ (or $(\!(\overline{a},k)\!)$ and $(\!(k,a)\!)$) is given by the number of vertices on the connected component containing $a$ of the tree without vertex $k$ (and the reflections that appear use those vertices and $k$ itself); and
\item the horizontal distance between clockwise adjacent neighbors of $k$ is given by the length of the \emph{run-leaf} between those neighbors (defined in~\Cref{def:run_leaves}).%5 2 3 3 4 4 10 6 7 8 6 8 9 5 9 10 2 7
\end{itemize}

%The successive products of the subword $\u$ from~\Cref{fig:main} acting on $k = 1$ are illustrated below, along with the corresponding tree from~\Cref{fig:main}.%in~\Cref{fig:prods}.  Compare with the tree also drawn in~\Cref{fig:main} (reproduced below).

\[\raisebox{-.5\height}{
\begin{tikzpicture}
\node[draw,circle,inner sep=0pt,thick,minimum width=1em] (2) at (-1,1) {$2$};
\draw[-,darkblue] (-1,1.2) to (-1,1.3);
\node[darkblue] (l2) at (-1,1.4) {\tiny $9$};
\node[draw,circle,inner sep=0pt,thick,minimum width=1em] (9) at (-1,-1) {$9$};
\draw[-,darkblue] (-1,-1.2) to (-1,-1.3);
\node[darkblue] (l9) at (-1,-1.4) {\tiny $9$};
\node[draw,circle,inner sep=0pt,thick,minimum width=1em,black,fill=darkred!20,text opacity=1] (10) at (-1,0) {$10$};
\draw[-,darkblue] (-1+.16,.16) to (-1+.21,.21);
\node[darkblue] (l0) at (-1+.28,.28) {\tiny $3$};
\draw[-,darkblue] (-1+.16,-.16) to (-1+.21,-.21);
\node[darkblue] (l0) at (-1+.28,-.28) {\tiny $4$};
\draw[-,darkblue] (-1-.16,.16) to (-1-.21,.21);
\node[darkblue] (l0) at (-1-.28,.28) {\tiny $2$};
\draw[-,darkblue] (-1-.16,-.16) to (-1-.21,-.21);
\node[darkblue] (l0) at (-1-.28,-.28) {\tiny $0$};
\node[draw,circle,inner sep=0pt,thick,minimum width=1em] (5) at (0,0) {$5$};
\draw[-,darkblue] (0,-.2) to (0,-.3);
\node[darkblue] (l5) at (0,-.4) {\tiny $4$};
\draw[-,darkblue] (.14,.14) to (.21,.21);
\node[darkblue] (l5) at (.28,.28) {\tiny $3$};
\draw[-,darkblue] (-.14,.14) to (-.21,.21);
\node[darkblue] (l5) at (-.28,.28) {\tiny $2$};
\node[draw,circle,inner sep=0pt,thick,minimum width=1em] (7) at (0,1) {$7$};
\draw[-,darkblue] (0,1.2) to (0,1.3);
\node[darkblue] (l7) at (0,1.4) {\tiny $9$};
\node[draw,circle,inner sep=0pt,thick,minimum width=1em] (1) at (1,0) {$1$};
\draw[-,darkblue] (1.14,.14) to (1.21,.21);
\node[darkblue] (l5) at (1.28,.28) {\tiny $1$};
\draw[-,darkblue] (1.14,-.14) to (1.21,-.21);
\node[darkblue] (l5) at (1.28,-.28) {\tiny $4$};
\draw[-,darkblue] (.86,.14) to (.79,.21);
\node[darkblue] (l5) at (.72,.28) {\tiny $2$};
\draw[-,darkblue] (.86,-.14) to (.79,-.21);
\node[darkblue] (l5) at (.72,-.28) {\tiny $2$};
\node[draw,circle,inner sep=0pt,thick,minimum width=1em] (3) at (1,1) {$3$};
\draw[-,darkblue] (1,1.2) to (1,1.3);
\node[darkblue] (l3) at (1,1.4) {\tiny $9$};
\node[draw,circle,inner sep=0pt,thick,minimum width=1em] (8) at (1,-1) {$8$};
\draw[-,darkblue] (1.2,-1) to (1.3,-1);
\node[darkblue] (l8) at (1.4,-1) {\tiny $7$};
\draw[-,darkblue] (.8,-1) to (.7,-1);
\node[darkblue] (l8) at (.6,-1) {\tiny $2$};
\node[draw,circle,inner sep=0pt,thick,minimum width=1em] (6) at (1,-2) {$6$};
\draw[-,darkblue] (1,-2.2) to (1,-2.3);
\node[darkblue] (l6) at (1,-2.4) {\tiny $9$};
\node[draw,circle,inner sep=0pt,thick,minimum width=1em] (4) at (2,0) {$4$};
\draw[-,darkblue] (2.2,0) to (2.3,0);
\node[darkblue] (l4) at (2.4,0) {\tiny $9$};
\draw[->,thick] (2) -- (10);
\draw[->,thick] (9) -- (10);
\draw[->,thick] (5) -- (10);
\draw[->,thick] (7) -- (5);
\draw[->,thick] (1) -- (5);
\draw[->,thick] (3) -- (1);
\draw[->,thick] (4) -- (1);
\draw[->,thick] (8) -- (1);
\draw[->,thick] (6) -- (8);
\draw[->,dashed,ultra thick,darkred] plot [smooth] coordinates {(-1.24,0) (-1.5,0) (-1.6,.3) (-1.6,.5)};
\end{tikzpicture}} \hspace{1em} \scalebox{0.6}{$\begin{NiceArray}{|>{\centering\arraybackslash$} p{2.5em} <{$}>{\centering\arraybackslash$} p{2.5em} <{$}>{\centering\arraybackslash$} p{2.5em} <{$}>{\centering\arraybackslash$} p{2.5em} <{$}>{\centering\arraybackslash$} p{2.5em} <{$}>{\centering\arraybackslash$} p{2.5em} <{$}>{\centering\arraybackslash$} p{2.5em} <{$}>{\centering\arraybackslash$} p{2.5em} <{$}>{\centering\arraybackslash$} p{2.5em} <{$}|} \hline\cc (\!(\overline{\ten}1)\!)& & & & & & & & 
\\ \cc (\!(\overline{2}1)\!)& & & & & & & & 
\\ \cc (\!(\overline{7}1)\!)& & & & & & & & 
\\ \cd (\!(\overline{5}1)\!)& \cc (\!(\overline{2}1)\!)& (\!(13)\!)& & & & & & 
\\ & & \cd (\!(\overline{3}1)\!)& (\!(14)\!)& & & & & 
\\ & & & \cd (\!(\overline{4}1)\!)& \cc (\!(\overline{\ten}1)\!)& \cc (\!(16)\!)& \cc (\!(\overline{7}1)\!)& (\!(18)\!)& 
\\ & & & & & & & \cc (\!(16)\!)& 
\\ & & & & & & & \cd (\!(\overline{8}1)\!)& \cc (\!(19)\!)
\\ (\!(15)\!)& & & & & & & & 
\\ \cc (\!(19)\!)& & & & & & & & \\ \hline \CodeAfter 
  
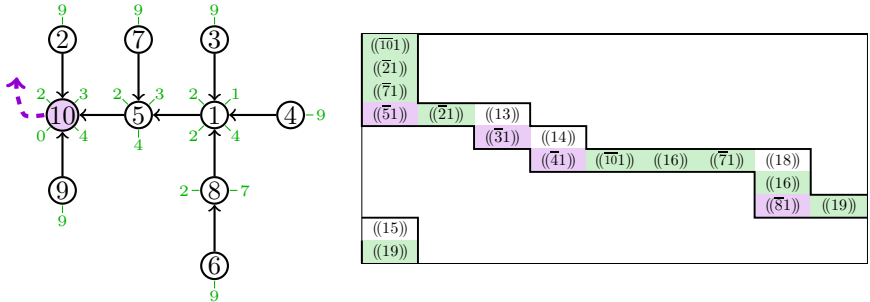
\tikz \draw [very thick] (1-|1) -- (5-|1) -- (5-|3) -- (6-|3) -- (6-|4) -- (7-|4) -- (7-|8) -- (9-|8) -- (9-|10); \tikz \draw [very thick] (1-|2) -- (4-|2) -- (4-|4) -- (5-|4) -- (5-|5) -- (6-|5) -- (6-|9) -- (8-|9) -- (8-|10); \tikz \draw [very thick] (9-|1) -- (9-|2) -- (11-|2); \tikz \draw [very thick] (10-|1) -- (11-|1);
\end{NiceArray}$}\]

The path above describes a clockwise turn around the vertex $1$ of the tree above (reproduced from~\Cref{fig:main}: starting at the corner $(\!(\overline{5}1)\!)$ corresponding to the neighbor $5$, there is a run-leaf with label 2 corresponding to a horizontal step of length 2, then the neighbor $3$ is visited as the corners $(\!(13)\!)$ and $(\!(\overline{3}1)\!)$ (there is a single vertical step since $3$ only has $1$ as a neighbor), there is a run-leaf with label 1 corresponding to one horizontal step,  the neighbor $4$ is visited as the corners $(\!(14)\!)$ and $(\!(\overline{4}1)\!)$ (again, there is a single vertical step since $4$ only has $1$ as a neighbor), then the run-leaf with label 4 gives four horizontal steps, we visit the neighbor $8$ as the corners $(\!(18)\!)$ and $(\!(\overline{8}1)\!)$ with two vertical steps between them because $8$ is connected to $6$, then a run-leaf with label 2, and then we revisit the neighbor 5 as the corner $(\!(15)\!)$ and we must make five vertical steps to return to the corner $(\!(\overline{5}1)\!)$ because of the vertices $5,7,10,2,9$. 
\end{example}

%The next two results will build towards showing that $\rru$ satisfies \cref{def:treelike_factorizations}.

\begin{lemma}
\label{maxDist}
    If $\u$ is an $e$-subword of $\blambda_n$, then for all integers $k$ and all $a, b$ with $0 \leq b - a < n(n-1)$,
    \[
        \big|\left(u_{a}\cdots u_{b}\right)(k) - k\big| \leq n-2,
    \]
    where the indices are taken modulo $n(n-1)$.
\end{lemma}
\begin{proof}

Any segment $\left(u_a\cdots u_{b}\right)$ and any $k$ such that \[\left(u_a\cdots u_{b}\right)(k) - k \geq n-1\] can be converted to the segment $\left(u_{b+1}\cdots u_{a-1}\right)$ and $k'=\left(u_a \cdots u_b\right)(k)$ with \[\left(u_{b+1}\cdots u_{a-1}\right)(k') - k' \leq 1-n,\] since
    \[
        \left(u_{b+1}\cdots u_{a-1}\right)\big(\left(u_a \cdots u_b\right)(k)\big) - \left(u_a \cdots u_b\right)(k) = k - \left(u_a \cdots u_b\right)(k) \leq 1-n.
    \]
    So it suffices to consider the case $\left(u_a\cdots u_{b}\right)(k) - k \leq 1-n$, so that the number of terms in $u_a\cdots u_{b}$ that decrease $k$ (by one, since each $u_i$ is a simple reflection) is at least $n-1$.    Moreover, we can take the index $b$ to be $n(n-1)$ by rotating the subword.  The terms that decrease $k$ correspond to indices $a \leq i_{n-1} < \cdots < i_1 \leq n(n-1)$ such that $u_{i_\ell} = s_{k - \ell}$. Notice that $i_{n-1} \leq (n-1)^2 + 1$.
    
    Since $\left(u_1 \cdots u_{a-1}u_a\cdots u_{b}\right)(k) = k$, there must exist at least $n-1$ terms in $(u_1\cdots u_{a-1})$ that increase $(u_a \cdots u_b)(k)$, so we can find indices $1 \leq j_1 \leq \cdots \leq j_{n-1} \leq a-1$ such that $u_{j_\ell} = s_{k-\ell}$. Each of these indices must be on its own row and not on the last two rows (because $j_{n-1} \leq i_{n-1} - n \leq (n-2)(n-1)$. So we reach a contradiction because $\blambda_n$ only has $n$ rows.
\end{proof}

\begin{lemma}
\label{smallMoves}
    If $\u$ is an $e$-subword of $\blambda_n$ and $(\!(a,b)\!) \in \Inv(\u)$, then $|b - a| \leq n-1$.
\end{lemma}
\begin{proof}
    We need to show that
    \[
        \left|\left(u_1\cdots u_{j}\right)(j+1) - \left(u_1\cdots u_{j}\right)(j)\right| \leq n-1
    \]
    for all $j = 0,\dots,n(n-1)-1$. Suppose not. Since each inversion is an affine reflection, it is not possible to have
    \[
        |\left(u_1\cdots u_{j}\right)(j+1) - \left(u_1\cdots u_{j}\right)(j)| = n,
    \]
    so suppose that
    \[
        \left|\left(u_1\cdots u_{j}\right)(j+1) - \left(u_1\cdots u_{j}\right)(j)\right| \geq n+1.
    \]
    By \cref{maxDist} there are two cases to consider:
    \begin{enumerate}
        \item $\left(u_1\cdots u_{j}\right)(j+1) \leq j < j + 1 \leq \left(u_1\cdots u_{j}\right)(j)$
        \item $\left(u_1\cdots u_{j}\right)(j) \leq j < j + 1 \leq \left(u_1\cdots u_{j}\right)(j+1)$
    \end{enumerate}
    In the first case, there must exist an index $1 < \ell < j$ such that $u_\ell = s_{\ell - 1}$ and
    \[
        \left(u_{\ell + 1} \cdots u_j\right)(j) = \ell-1 \quad \text{and} \quad \left(u_{\ell + 1} \cdots u_j\right)(j+1) = \ell.
    \]
    Choose the smallest such $\ell$. Then
    \[
        \left|\left(u_{1}\cdots u_{\ell-1}\right)(\ell-1) - \left(u_{1}\cdots u_{\ell-1}(\ell)\right)\right| = \left|\left(u_1\cdots u_{j}(j+1)\right) - \left(u_1\cdots u_{j}\right)(j)\right| \geq n+1,
    \]
    and
    \[
        \left(u_1\cdots u_{\ell-1}\right)(\ell - 1) \leq \ell-1 < \ell \leq \left(u_1\cdots u_{j}\right)(\ell).
    \]
    So the first case reduces to the second.  Let
    \[
        c = j -  \left(u_1\cdots u_{j}\right)(j) \quad \text{and} \quad d = \left(u_1\cdots u_{j}\right)(j+1) - (j+1),
    \]
    so $c + d \geq n$. Now there must be indices $1 \leq i_d < i_{d-1} \cdots < i_1 \leq j$ such that $u_{i_\ell} = s_{j + \ell}$. Each of these indices must be on its own row. Moreover, $i_1$ cannot be on the same row as $j$.

    Now there must also be indices $j+1 \leq k_1 < \cdots < k_c \leq n(n-1)$ such that $u_{k_\ell} = s_{j - \ell}$. Again, each of these indices must be its own row. Moreover, $k_1$ cannot be on the same row as $j$.  It follows that we need to use at least $c + d + 1 \geq n + 1$ rows, which is a contradiction.
\end{proof}

\begin{lemma}
\label{lem:subCond1}
    For $\u \in \S_n$, $\rru$ is a tree-like factorization in $\EFt_n$.% write $\rru = \left[t_1,\dots,t_{2n-2}\right]$. Then $t_1\cdots t_{2n-2} \in \Fact(\lambda_n)$ satisfies condition (1) of \cref{def:cyclicFact}.
\end{lemma}
\begin{proof} Write $\rru = \left[r_1,\dots,r_{2n-2}\right]$. If the skips corresponding to $r_\ell$ and $r_{\ell + 1}$ are at indices $i$ and $j$, respectively, then
    \[
        u_{(j-1)}(j-1) = u_{(i-1)} e s_{i}\cdots s_{j-2}(j-1) = u_{(i-1)}(i).
    \]
    Since $r_\ell = (\!(u_{(i-1)}(i-1), u_{(i-1)}(i) )\!)$ and $r_{\ell + 1} = (\!(u_{(j-1)}(j-1), u_{(j-1)}(j) )\!)$, it follows that there exist $a_0,\dots a_{2n-2} \in \mathbb{Z}$ such that $r_\ell = (\!(a_{\ell - 1} , a_\ell)\!)$.     The result then follows from \cref{smallMoves} and \cref{connection}.
\end{proof}

We now show that~\Cref{def:subwords} captures the usual definition of \emph{distinguished} subword, that a simple reflection must be used if it causes the current product to decrease in weak order.

\begin{corollary}
\label{cor:distinguished}
  For $\u \in \S_n$, if $u_{(j)}s_j < u_{(j)}$ then $u_{j+1} = s_j$.
\end{corollary}
\begin{proof}
It is clear that $u_{(j)}s_j < u_{(j)}$ iff $u_{(j)}(j) > u_{(j)}(j+1)$.  Suppose that $u_{j+1} = e$, so the inversion corresponding to this skip is $(\!(u_{(j)}(j), u_{(j)}(j+1) )\!)$. By the proofs of \Cref{lem:subCond1} and \Cref{connection}, it follows that $u_{(j)}(j) < u_{(j)}(j+1)$.  The result follows by contraposition.
\end{proof}

\subsection{Skip reflections are cyclic}
We continue to work towards the bijection between $\S_n$ and $\E_n$ by now showing that $\rru$ is actually a cyclic factorization in $\Ft_n$.

\begin{lemma}
\label{locSkip}
    For $\u$ a subword of $\blambda_n$, write $\rru = \left[r_1,\dots,r_{\ell}\right]$, and let $i_j$ denote the index of the skip corresponding to $r_j$. Then
    \[
        (r_1\cdots r_{j-1})r_j(r_{j-1}\cdots r_1) = \left(\!\left(0 , i_j + \left\lfloor \frac{i_j-1}{n-1} \right\rfloor\right)\!\right).
    \]
\end{lemma}
\begin{proof}
    From Lemma \ref{invProd}, we have
    \begin{align*}
        (r_1\cdots r_{j-1})r_j(r_{j-1}\cdots r_1) &= (\lambda_{n})_{(i_{j}-1)} u_{(i_{j}-1)}^{-1}(u_{(i_{j}-1)} s_{i_j-1} u_{(i_{j}-1)}^{-1}) u_{(i_{j}-1)} (\lambda_{n})_{(i_{j}-1)}^{-1} \\
        &= (s_0 \cdots s_{i_j-2}) s_{i_j-1} (s_{i_j-2} \cdots s_0) \\
        &= (\!( s_0 \cdots s_{i_j-2}(i_j - 1) , s_0 \cdots s_{i_j-2}(i_j) )\!) \\
        &= (\!( 0 , s_0 \cdots s_{i_j-2}(i_j) )\!).
    \end{align*}
    Write $a_k \coloneqq s_0 \cdots s_{k-2}(k)$. If $k < n$, then $a_k = k$. Otherwise,
    \begin{align*}
        a_k &= s_0\cdots s_{k - n}(k) = s_0 \cdots s_{k - n - 1}(k + 1)\\ &= s_0 \cdots s_{k - n - 1}(k - n + 1) + n = a_{k-(n-1)} + n.
    \end{align*}
    This shows that $a_k = k + \lfloor (k-1)/(n-1) \rfloor$, so the result follows.
\end{proof}

\begin{corollary}
    For $\u \in \S_n$, $\rru$ is a cyclic factorization in $\Ft_n$.
\end{corollary}
\begin{proof}
Write $\rru = \left[r_1,\dots,r_{2n-2}\right]$. 
    From \Cref{lem:subCond1}, we know $\left[r_1,\ldots, r_{2n-2}\right]$ satisfies \Cref{def:treelike_factorizations}, so it remains to show \crefitem{def:cyclic_factorizations}{CF3} and \crefitem{def:cyclic_factorizations}{CF2}. These follow from \cref{locSkip} and \cref{incClock}.
\end{proof}

\begin{proposition}
Let $\rr \in \Ft_n$.  Then there exists a subword $\u \in \S_n$ such that $\rr=\rru$.
\end{proposition}
\begin{proof}
Suppose that $\left[r_1,\ldots, r_{2n-2}\right]$ is a cyclic factorization of $\lambda_n$.  By \cref{nest}, we can write $r_\ell = (\!(a_{\ell - 1}, a_\ell )\!)$ with $a_{\ell-1} < a_\ell$, $a_0 = 0$, and $a_{2n-2} = n(n-1)$.

As in~\Cref{incClock}, for $1 \leq j \leq 2n-2$, define integers $m_j$ by
\[
        (\rrr_1\cdots \rrr_{j-1})\rrr_j(\rrr_{j-1}\cdots \rrr_1) = (\!(0,m_j)\!).
    \]
Then 
\[
    m_{2n-2} = a_{2n-1} + n < a_{2n-2}+n = n(n-1) + n = n^2.
\]
By \cref{incClock}, we have
\[
    0 = a_0 < a_1 = m_1 < m_2 < \cdots < m_{2n-2} < n^2.
\]
For $j = 1,\dots, 2n-2$, define 
\[
    i_j \coloneqq m_j - \left\lfloor \frac{m_j - 1}{n} \right\rfloor \text{ so that }  m_j = i_j - \left\lfloor \frac{i_j - 1}{n-1}\right\rfloor.
\]
%\color{red} could also use $m_j$ in place of $m_j-1$ since $m_j$ is never a multiple of $n$. Not sure which is better, if you want to prove relation between $i_j$ and $m_j$... \color{black}
%so that
%\[
%   
%\]
Notice that
\[
    1 \leq i_j \leq n^2 - 1 - \left\lfloor \frac{n^2 - 2}{n}\right\rfloor = n(n-1),
\]
so we can take $\u$ to be the subword of $\lambda_n$ with skips at indices $i_j$. By \cref{locSkip}, we have $\rru = \left[r_1,\dots, r_{2n-2}\right]$. Then, since $r_1 \cdots r_{2n-2} = \lambda_n$, it follows from \cref{invProd} that $\u$ is a distinguished subword.
\end{proof}

\begin{corollary}
The map $\u \mapsto \rru$ is a bijection between $\S_n$ and $\Ft_n$.
\end{corollary}

\section{Subwords and cyclic trees}
\label{sec:bijection}

In this section, we prove our main theorem---that there is a bijection between maximal distinguished subwords of $\blambda_n$ and (cyclically-embedded) vertex-labeled trees with $n$ vertices.

\begin{theorem}
\label{thm:main_thm}
There is a bijection between $\S_n$ and $\E_n$.
\end{theorem}

The proof of~\Cref{thm:main_thm} will occupy the next two subsections.

\subsection{From subwords to cyclic trees}
The forward direction of the bijection is easy: given $\u \in \S_n$, compute the inversions of the skips $\rru$, then create a tree $T \in \E_n$ with edges $(a,b)$ between $a<b$ when $(\!(a,b)\!)$ and $(\!(\overline{b},a)\!)$ appear as reflections in $\rru$.  This tree can then be cyclically embedded using~\Cref{sec:cyclic_from_trees}.

%straightforward. To make it easier, we decorate the trees with \textit{run-leaves} which aid in the construction of the corresponding subword in $\S_n$.

\subsection{From cyclic trees to subwords}
\label{def:run_leaves}

The other direction of the bijection is a little more difficult.  To more easily describe it, we decorate the trees with \emph{run-leaves}.% which aid in the construction of the corresponding subword in $\S_n$.

Fix a cyclic tree $T \in \E_n$. % Given a vertex $v \neq n$ in our original tree, 
To each vertex $v \neq n$ we will attach $\deg(v)$ many \defn{run-leaves}, so that in a clockwise walk around $T$, edges and run-leaves alternate.  At the vertex $n$, we instead add $\deg(n) + 1$ many leaves: two between the smallest and largest neighbors of $n$ (because $T$ is cyclically embedded, these vertices will be adjacent).   We index the run-leaves based on when we see them in the clockwise walk starting from $n$ towards its smallest neighbor, so that our walk visits run-leaves $l_0,\ldots,l_{2n-2}$ and edges $e_0,\ldots,e_{2n-1}$ in the order
\begin{equation}
\label{eq:takes_and_skips}
\left[l_0, e_0,l_1,e_1,l_2,\ldots,e_{2n-1},l_{2n-2}\right].
\end{equation}

We now label each run-leaf $l_k$ with an integer $\ell(l_k) \coloneqq \ell_k$ with $1 \leq \ell_k \leq n-1$ as follows (for now, ignore the first and last run-leaves, $l_0$ and $l_{2n-2}$, attached to vertex $n$).  The run-leaf $l_k$ is situated between the two edges $e_{k-1}=(v_{k-1},v_k)$ and $e_k=(v_k,v_{k+1})$, incident to the vertex $v_k$ to which $l_k$ has been attached.  The label $\ell_k$ is assigned according to the following four cases, illustrated in~\Cref{fig:runleaflabels}:

\begin{figure}[htbp]
\[\begin{array}{|lc|cr|} \hline & & & \\[-1ex] (a) &
\raisebox{-.5\height}{\begin{tikzpicture}
%\node (l) at (-2,.5) {(a)};
\node[draw,rounded corners=2pt,inner sep=3pt,thick,minimum width=1ex] (-1) at (-1,0) {$v_{k-1}$};
\node[draw,rounded corners=2pt,inner sep=3pt,thick,minimum width=1ex] (0) at (0,0) {$v_{k}$};
\node[draw,rounded corners=2pt,inner sep=3pt,thick,minimum width=1ex] (1) at (1,0) {$v_{k+1}$};
\node[darkblue] (l) at (0,.55) {\tiny $\ell_{k}$};
\draw[-,darkblue] (0,.2) to (0,.4);
\draw[<-,thick] (-1) to (0);
\draw[<-,thick] (0) to (1);
\draw[->,dashed,ultra thick, darkred] plot [smooth] coordinates {(-1.5,.4) (-1,0.4) (0,.7) (1,0.4) (1.5,.4)};
\end{tikzpicture}} & 
\raisebox{-.5\height}{\begin{tikzpicture}
%\node (l) at (-1,0) {(b)};
\node[draw,rounded corners=2pt,inner sep=3pt,thick,minimum width=1ex] (-1) at (1,.5) {$v_{k-1}$};
\node[draw,rounded corners=2pt,inner sep=3pt,thick,minimum width=1ex] (0) at (0,0) {$v_{k}$};
\node[darkblue] (l) at (.7,0) {\tiny $\ell_{k}$};
\node[draw,rounded corners=2pt,inner sep=3pt,thick,minimum width=1ex] (1) at (1,-.5) {$v_{k+1}$};
\draw[-,darkblue] (.3,0) to (.5,0);
\draw[->,thick] (-1) to (0);
\draw[<-,thick] (0) to (1);
\draw[->,dashed,ultra thick, darkred] plot [smooth] coordinates {(1.9,0.7) (1.8,.3) (1,.1) (1,-.1) (1.8,-.3) (1.9,-0.7)};
\end{tikzpicture}} & (b)
\\ &
 v_{k+1} = v_k + {\color{darkblue}{\ell_k}} \ (\mathrm{mod}\ n-1) &
v_{k+1} = v_{k-1} + {\color{darkblue}{\ell_k}} \ (\mathrm{mod}\ n-1) &
 \\ \hline 
 & & & \\[-1ex] (c) &
 \raisebox{-.5\height}{\begin{tikzpicture}
\node[draw,rounded corners=2pt,inner sep=3pt,thick,minimum width=1ex] (-1) at (-1,0) {$v_{k+1}$};
\node[draw,rounded corners=2pt,inner sep=3pt,thick,minimum width=1ex] (0) at (0,0) {$v_{k}$};
\node[darkblue] (l) at (0,-.5) {\tiny $\ell_{k}$};
\node[draw,rounded corners=2pt,inner sep=3pt,thick,minimum width=1ex] (1) at (1,0) {$v_{k-1}$};
\draw[-,darkblue] (0,-.2) to (0,-.4);
\draw[<-,thick] (-1) to (0);
\draw[<-,thick] (0) to (1);
\draw[<-,dashed,ultra thick, darkred] plot [smooth] coordinates {(-1.5,-.4) (-1,-0.4) (0,-.7) (1,-0.4) (1.5,-.4)};
\end{tikzpicture}} &
\raisebox{-.5\height}{\begin{tikzpicture}
\node[draw,rounded corners=2pt,inner sep=3pt,thick,minimum width=1ex] (0) at (0,0) {$v_{k-1}$};
\node[draw,rounded corners=2pt,inner sep=3pt,thick,minimum width=1ex] (1) at (1,0) {$v_{k}$};
\node[darkblue] (l) at (1.7,0) {\tiny $\ell_{k}$};
\draw[-,darkblue] (1.3,0) to (1.5,0);
\draw[<-,thick] (0) to (1);
\draw[->,dashed,ultra thick, darkred] plot [smooth] coordinates {(-.5,.4) (0,0.4) (1.8,0.5) (1.8,-0.5) (0,-0.4) (-.5,-.4)};
\end{tikzpicture}} & (d)
\\
& v_{k} = v_{k-1} + {\color{darkblue}{\ell_k}} \ (\mathrm{mod}\ n-1) &
{\color{darkblue}{\ell_k}}=n-1 &
\\ \hline
\end{array}\]
\caption{Run-leaf rules, where the dashed line denoting our walk around a cyclic tree $T$. The walk begins at vertex $n$ and steps first towards the smallest neighbor of $n$ walking clockwise.}
\label{fig:runleaflabels}
\end{figure}
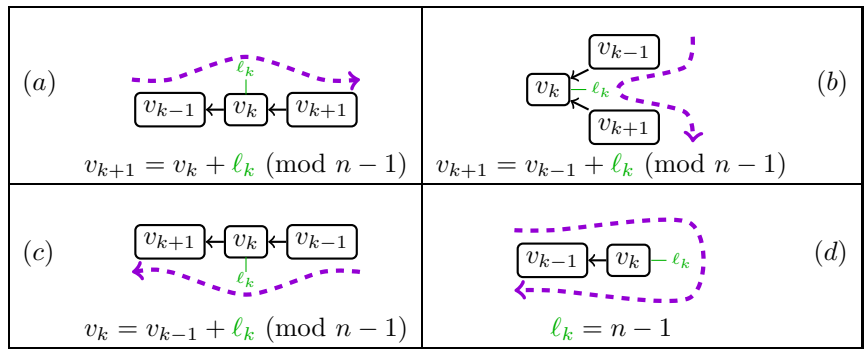

\begin{enumerate}[label=(\alph*)]
\item if $v_{k+1}\neq v_{k-1}$ and the path from $v_{k+1}$ to $n$ goes through $v_k$ and $v_{k-1}$, then $\ell_k=v_{k+1}-v_{k} \mod n-1$.
\item if the paths from $v_{k-1}$ to $n$ and from $v_{k+1}$ go through $v_k$, then $\ell_k=v_{k+1}-v_{k-1} \mod n-1$.
\item if $v_{k+1}\neq v_{k-1}$ and the path from $v_{k-1}$ to $n$ goes through $v_k$ and $v_{k+1}$, then $\ell_k=v_{k}-v_{k-1} \mod n-1$.
\item if $v_{k+1}= v_{k-1}$, then $\ell_k = n-1$.
%If neither $v_{k-1}$ nor $v_{k+1}$ lies on the path from $v_k$ to $n$, we will set $\ell_k = v_{k+1} - v_{k-1} \mod (n-1)$. If one of these vertices is on the path from $v_k$ to $n$, then we use $v_k$ in its place in the equation. In particular, if both $v_{k-1}$ and $v_{k+1}$ are on the path, then we must have $v_{k-1}=v_{k+1}$ since our graph is a tree. In this case, we set $\ell_k = n-1$.
\end{enumerate}

We can view these four cases as specializations of the general rule for $1\leq \ell_k \leq n-1$:
\begin{align*}
\ell_k &= v'_{k+1} - v'_{k-1} \mod (n-1), \text{ where } \\
v'_{k\pm 1}&=\begin{cases} v_{k} & \text{if }v_{k\pm 1}\text{ is on the path from }v_k\text{ to }n, \\ v_{k\pm 1} & \text{ otherwise.} \end{cases}
\end{align*}

Finally, we define $\ell_0$ to be $v_1$, the smallest neighbor of $n$, and $\ell_{2n-2}$ to be $(n-1)-v_{2n-3}$; note that $v_{2n-3}$ is necessarily the largest neighbor of $n$ and $\ell_{2n-2}$ is possibly zero.  By construction, the sum of the labels of the run-leaves adjacent to any vertex $k \in [n]$ is $n-1$.
%In each case, we choose the representative $1\leq \ell_k \leq n-1$.

\begin{remark}
Because there is a unique path from $v_k$ to $n$, we never have the case \raisebox{-.5\height}{\begin{tikzpicture}
%\node (l) at (-1,0) {(b)};
\node[draw,rounded corners=2pt,inner sep=3pt,thick,minimum width=1ex] (-1) at (1,.5) {$v_{k-1}$};
\node[draw,rounded corners=2pt,inner sep=3pt,thick,minimum width=1ex] (0) at (0,0) {$v_{k}$};
\node[darkblue] (l) at (.7,0) {\tiny $\ell_{k}$};
\node[draw,rounded corners=2pt,inner sep=3pt,thick,minimum width=1ex] (1) at (1,-.5) {$v_{k+1}$};
\draw[-,darkblue] (.3,0) to (.5,0);
\draw[<-,thick] (-1) to (0);
\draw[->,thick] (0) to (1);
\draw[->,dashed,ultra thick, darkred] plot [smooth] coordinates {(1.9,0.7) (1.8,.3) (1,.1) (1,-.1) (1.8,-.3) (1.9,-0.7)};
\end{tikzpicture}}.
\end{remark}

%---if $k \neq n$, then let $v_1<\cdots<k<\cdots<v_i$ be the neighbors of $k$ listed in clockwise order, where the neighbor on the path from $k$ to $n$ has been replaced by $k$ (if $k$ is a leaf, then this list contains only $k$).  Then the run-leaves are labeled $v_2-v_1 \mod (n-1), v_3-v_2 \mod (n-1),\ldots,v_1-v_i \mod (n-1)$.  Since we have arranged the vertices in increasing order, these labels are actually $v_2-v_1, v_3-v_2,\ldots,v_1-v_i + (n-1)$---so the sum telescopes to $(n-1)$.  When $k=n$, a similar telescoping applies.%, just without any substitution on the neighbors.

The subword $\u^T \in \S_n$ is now described using the sequence of run-leaves and edges in~\Cref{eq:takes_and_skips} to describe its takes and skips (see~\Cref{sec:subwords} for these definitions): each run-leaf $l_k$ corresonds to a series of $\ell_k-1$ successive takes, while each edge $e_k$ corresponds to a single skip.  \cref{smallMoves}, \cref{nest}, and \cref{locSkip} justify this procedure.    Write \[\rru = \left[r_1,\dots,r_{2n-2}\right],\] and let $i_j$ be the index of the skip corresponding to $r_j$.  Consider $r_\ell = (\!(a,b)\!)$ and $r_{\ell+1} = (\!(b,c)\!)$ with $a < b < c$.  Then we have the following cases (each one corresponds to a case in~\Cref{fig:runleaflabels}):
    \begin{enumerate}[label=(\alph*)]
        \item If $r_\ell$ and $r_{\ell + 1}$ are both the left ends of their pairs, then 
        \begin{itemize}
            \item $i_{\ell + 1} - i_\ell = c - b$ if $(c \mod n) > (b \mod n)$
            \item $i_{\ell + 1} - i_\ell = c - b - 1$ if $(c \mod n) < (b \mod n)$
        \end{itemize}

        \item If $r_\ell$ is the right end of its pair and $r_{\ell + 1}$ is the left end of its pair, then 
        \begin{itemize}
            \item $i_{\ell + 1} - i_\ell = c - a$ if $(c \mod n) > (a \mod n)$
            \item $i_{\ell + 1} - i_\ell = c - a - 1$ if $(c \mod n) < (a \mod n)$
        \end{itemize}

        \item If $r_\ell$ and $r_{\ell + 1}$ are both the right ends of their pairs, then
        \begin{itemize}
            \item $i_{\ell + 1} - i_\ell = b - a$ if $(b \mod n) > (a \mod n)$
            \item $i_{\ell + 1} - i_\ell = b - a - 1$ if $(b \mod n) < (a \mod n)$
        \end{itemize}

        \item If $r_\ell$ is the left end of its pair and $r_{\ell + 1}$ is the right end of its pair, then $c = a+n$ and $i_{\ell + 1} - i_\ell = n-1$.
    \end{enumerate}
%\end{proposition}
%\begin{proof}
%    This is a straightforward computation using 
%\end{proof}

%and/or $v_{k-1}$ is replaced by $v_k$ if it is on the path between $v_k$ and $n$.
%Recall our rule for embedding where the neighbors of a vertex $v_k$ increase in clockwise order, with the special rule that the vertex lying on the path from $v_k$ to $n$ is labeled $v_k$. Accordingly, if neither $v_{k-1}$ nor $v_{k+1}$ lies on the path from $v_k$ to $n$, we will set $\ell_k = v_{k+1} - v_{k-1} \mod (n-1)$. If one of these vertices is on the path from $v_k$ to $n$, then we use $v_k$ in its place in the equation. In particular, if both $v_{k-1}$ and $v_{k+1}$ are on the path, then we must have $v_{k-1}=v_{k+1}$ since our graph is a tree. In this case, we set $\ell_k = n-1$. 

\begin{example}
Let $T$ be the cyclic tree from~\Cref{fig:main} (reproduced below):
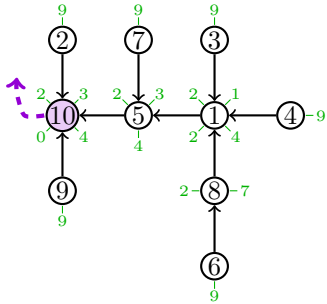
\[
\raisebox{-.5\height}{
\begin{tikzpicture}
\node[draw,circle,inner sep=0pt,thick,minimum width=1em] (2) at (-1,1) {$2$};
\draw[-,darkblue] (-1,1.2) to (-1,1.3);
\node[darkblue] (l2) at (-1,1.4) {\tiny $9$};
\node[draw,circle,inner sep=0pt,thick,minimum width=1em] (9) at (-1,-1) {$9$};
\draw[-,darkblue] (-1,-1.2) to (-1,-1.3);
\node[darkblue] (l9) at (-1,-1.4) {\tiny $9$};
\node[draw,circle,inner sep=0pt,thick,minimum width=1em,black,fill=darkred!20,text opacity=1] (10) at (-1,0) {$10$};
\draw[-,darkblue] (-1+.16,.16) to (-1+.21,.21);
\node[darkblue] (l0) at (-1+.28,.28) {\tiny $3$};
\draw[-,darkblue] (-1+.16,-.16) to (-1+.21,-.21);
\node[darkblue] (l0) at (-1+.28,-.28) {\tiny $4$};
\draw[-,darkblue] (-1-.16,.16) to (-1-.21,.21);
\node[darkblue] (l0) at (-1-.28,.28) {\tiny $2$};
\draw[-,darkblue] (-1-.16,-.16) to (-1-.21,-.21);
\node[darkblue] (l0) at (-1-.28,-.28) {\tiny $0$};
\node[draw,circle,inner sep=0pt,thick,minimum width=1em] (5) at (0,0) {$5$};
\draw[-,darkblue] (0,-.2) to (0,-.3);
\node[darkblue] (l5) at (0,-.4) {\tiny $4$};
\draw[-,darkblue] (.14,.14) to (.21,.21);
\node[darkblue] (l5) at (.28,.28) {\tiny $3$};
\draw[-,darkblue] (-.14,.14) to (-.21,.21);
\node[darkblue] (l5) at (-.28,.28) {\tiny $2$};
\node[draw,circle,inner sep=0pt,thick,minimum width=1em] (7) at (0,1) {$7$};
\draw[-,darkblue] (0,1.2) to (0,1.3);
\node[darkblue] (l7) at (0,1.4) {\tiny $9$};
\node[draw,circle,inner sep=0pt,thick,minimum width=1em] (1) at (1,0) {$1$};
\draw[-,darkblue] (1.14,.14) to (1.21,.21);
\node[darkblue] (l5) at (1.28,.28) {\tiny $1$};
\draw[-,darkblue] (1.14,-.14) to (1.21,-.21);
\node[darkblue] (l5) at (1.28,-.28) {\tiny $4$};
\draw[-,darkblue] (.86,.14) to (.79,.21);
\node[darkblue] (l5) at (.72,.28) {\tiny $2$};
\draw[-,darkblue] (.86,-.14) to (.79,-.21);
\node[darkblue] (l5) at (.72,-.28) {\tiny $2$};
\node[draw,circle,inner sep=0pt,thick,minimum width=1em] (3) at (1,1) {$3$};
\draw[-,darkblue] (1,1.2) to (1,1.3);
\node[darkblue] (l3) at (1,1.4) {\tiny $9$};
\node[draw,circle,inner sep=0pt,thick,minimum width=1em] (8) at (1,-1) {$8$};
\draw[-,darkblue] (1.2,-1) to (1.3,-1);
\node[darkblue] (l8) at (1.4,-1) {\tiny $7$};
\draw[-,darkblue] (.8,-1) to (.7,-1);
\node[darkblue] (l8) at (.6,-1) {\tiny $2$};
\node[draw,circle,inner sep=0pt,thick,minimum width=1em] (6) at (1,-2) {$6$};
\draw[-,darkblue] (1,-2.2) to (1,-2.3);
\node[darkblue] (l6) at (1,-2.4) {\tiny $9$};
\node[draw,circle,inner sep=0pt,thick,minimum width=1em] (4) at (2,0) {$4$};
\draw[-,darkblue] (2.2,0) to (2.3,0);
\node[darkblue] (l4) at (2.4,0) {\tiny $9$};
\draw[->,thick] (2) -- (10);
\draw[->,thick] (9) -- (10);
\draw[->,thick] (5) -- (10);
\draw[->,thick] (7) -- (5);
\draw[->,thick] (1) -- (5);
\draw[->,thick] (3) -- (1);
\draw[->,thick] (4) -- (1);
\draw[->,thick] (8) -- (1);
\draw[->,thick] (6) -- (8);
\draw[->,dashed,ultra thick,darkred] plot [smooth] coordinates {(-1.24,0) (-1.5,0) (-1.6,.3) (-1.6,.5)};
\end{tikzpicture}}.\]
The sequence of run-leaf labels and edges visited during the clockwise walk around $T$ is:
\begin{align*}
\big[2,(\ten 2),9,(2 \ten),3,(\ten 5),2,(57),9,(75),3,(51),2,(13),9,(31),1,(14),9,\\
(41),4,(18),7,(86),9,(68),2,(81),2,(15),4,(5\ten),4,(\ten 9),9,(9\ten),0\big].
\end{align*}
Replacing run-leaves by runs (green) and edges by skips (white), we obtain the subword from~\Cref{fig:main} (reproduced below):
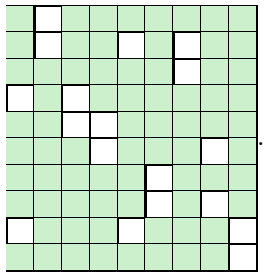
\renewcommand{\arraystretch}{.8}
\[\begin{array}{|c|c|c|c|c|c|c|c|c|} \hline
\cc 
 &  & \cc 
 & \cc 
 & \cc 
 & \cc  & \cc  & \cc  & \cc  \\\hline
 \cc  & & \cc 
 & \cc 
 & & \cc 
 &  & \cc  & \cc  \\\hline
 \cc  & \cc  & \cc 
 & \cc 
 & \cc 
 & \cc 
 & & \cc  & \cc  \\\hline
  & \cc  &  & \cc 
 & \cc 
 & \cc 
 & \cc 
 & \cc 
 & \cc  \\\hline
 \cc  & \cc  &  &  &
\cc 
 & \cc 
 & \cc 
 & \cc 
 & \cc 
 \\\hline
 \cc  & \cc  & \cc  &  & \cc  & \cc 
 & \cc 
 &  & \cc 
 \\ \hline
 \cc 
 & \cc  & \cc  & \cc  & \cc  & & \cc 
 & \cc 
 & \cc 
 \\\hline
 \cc 
 & \cc 
 & \cc  & \cc  & \cc  & & \cc  &
 & \cc 
 \\\hline
 & \cc 
 & \cc 
 & \cc  &  & \cc  & \cc & \cc  &
 \\\hline
 \cc 
 & \cc 
 & \cc 
 & \cc 
 & \cc  & \cc  & \cc  & \cc  &  \\ \hline
\end{array}.\]
\end{example}

%\begin{proposition}\label{prop:run_leaves_sum_to_n-1}
%\end{proposition}
%\begin{proof}
%\end{proof}

\section{Enumeration}
\label{sec:details}
In this section we prove Cayley's formula for the number of vertex-labeled trees:
\[\left|\E_n\right|=n^{n-2}.\]
Our proof uses the bijection between $\E_n$ and $\S_n$ from~\Cref{thm:main_thm}, along with representation-theoretic techniques (previously obtained in a collaboration between the last author with P.~Galashin and T.~Lam) to compute the number of points in a particular braid variety $R_{\blambda_n}(\FF_q)$ over the finite field $\FF_q$ with $q$ elements.  Using a trace formula due to Opdam and an identity due to Haglund, we obtain that
\[\left|R_{\blambda_n}(\FF_q)\right|=(q-1)^{2n-2} [n]_q^{n-2},\]
where $[n]_q\coloneqq\frac{q^{n}-1}{q-1}$ is the usual $q$-analogue.
Certain \emph{distinguished subwords} $\mathcal{D}_{\blambda_n}$ index the Deodhar components of this braid variety $D_\u(\FF_q)$:
\[R_{\blambda_n}(\FF_q) = \bigsqcup_{\u \in \mathcal{D}_{\blambda_n}} D_\u(\FF_q),\]
but the maximal distinguished subwords in $\S_n$ are the only components that contribute to the sum when $q$ is sent to $1$:
\[|\S_n| = \left((q-1)^{-2n+2} \left|R_{\blambda_n}(\FF_q)\right|\right)\Bigg|_{q\to 1} = n^{n-2}.\]

\subsection{Braid varieties}
\label{sec:braid_varieties}
The usual definition of braid varieties extends to the context of Kac-Moody groups (for simplicity, we give a specialization of the more general definition).  For $\w = [s_1,s_2,\ldots, s_m]$ a word in the simple reflections $S$ of the Weyl group $W$, we denote this \defn{braid variety} (over a finite field) by $R_{\w}(\FF_q)$.  In slightly more detail, a split minimal Kac-Moody group $G$ is associated to a symmetrizable generalized Cartan matrix; it is generated by a split torus $T$ and root subgroups $U^\pm = \{U_{\pm \alpha_i}\}$.  We have opposite Borel subgroups $B^\pm$ generated by $T$ and $U^\pm$, we have the flag variety $\mathcal{B} = G/B^+$ with its decomposition into Schubert cells $\mathcal{B}_w = B^+ \cdot{w} B^+/B^+$ and opposite Schubert cells $\mathcal{B}^w = B^- \cdot{w} B^+/B^+$, and we can speak of the \defn{relative position} of two flags $B_1,B_2 \in G/B^+$ (written $B_1 \xrightarrow{w} B_2$ for $w \in W$).  Then for $\w = [s_1,s_2,\ldots, s_m]$ with $w = s_1s_2 \cdots s_m \in W$, we have
\[R_{\w}(\FF_q) = \left\{ B^+=B_0 \xrightarrow{s_1} B_1 \xrightarrow{s_2} B_2 \cdots \xrightarrow{s_m} B_m : B_m \in \mathcal{B}^e \right\}.\]

A \defn{distinguished subword} $\u$ of $\w$ is a subword for which a simple reflection must be used if it causes the current product to decrease in weak order---that is, if $u_{(j)}s_j < u_{(j)}$ then $u_{j+1} = s_j$ (see also~\Cref{cor:distinguished}).  Write $\mathcal{D}_{\w}$ for all distinguished subwords.  By a natural extension of~\cite{deodhar1985some} (see also~\cite{bao2021flag}), the braid variety $R_\w(\FF_q)$ has a Deodhar decomposition into 
\[R_{\w}(\FF_q) = \bigsqcup_{\u \in \mathcal{D}_{\w}} D_\u(\FF_q),\] where each $D_\u(\FF_q)$ is isomorphic to $(\FF_q^\times)^{s(\u)} \times \FF_q^{t(\u)} $, where $s(\u)$ is the number of skips of $\u$ and $t(\u)$ is half the number of takes.

Let $B_W$ be the \defn{braid group} for $W$ with generators $T_i$ for each $s_i \in S$, and let $H_W = B_W/(T_i^2=(q-1)T_i+q)$  be the \defn{Hecke algebra}, with usual basis $\{T_w\}_{w\in W}$.  Write $w=s_1\cdots s_m \in W$.  By the same arguments as~\cite[Lemmas A3 and A4]{kazhdan1979representations} and~\cite[Corollary 5.3]{galashin2022rational} the number of $\mathbb{F}_q$-points in the braid variety $R_{\w}(\FF_q)$ is given by the trace \[\left| R_{\w}(\FF_q)\right| = q^{\ell(w)} \tr( T_{\w}^{-1} ),\] where for $X \in H_W$, $\tr(X)$ returns the coefficient of the basis element $T_e$ indexed by the identity.

\subsection{Opdam's trace formula}
We now specialize to $G$ the affine Kac-Moody group of type $A_{n-1}$. Write $\Phi^+$ for the positive roots of $\mathrm{GL}_n$, $Q = \bigoplus_{i=1}^\rank \mathbb{Z} \alpha_i$ for its \emph{root lattice}, $Q^+ \subset Q$ for the positive span of the simple roots, and $\Lambda$ for the weight lattice.  Write $\widehat{S}_n$ for the \defn{extended affine symmetric group}, whose elements can be thought of as bijections $\hw: \mathbb{Z} \to \mathbb{Z}$ such that $\hw(i+n)=\hw(i)+n$ and $\sum_{i=1}^n \hw(i) = \binom{n+1}{2} \mod n$; it contains the elements of $\Lambda$ as translations.

Given $\lambda \in Q^+$, we express $\lambda$ in the basis of fundamental weights as $\lambda = \sum_{i=1}^{n-1} a_i \lambda_i$ and define $\lambda_+ = \sum_{i : a_i > 0} a_i \lambda_i$ and $\lambda_- = -\sum_{i : a_i < 0} a_i \lambda_i$.   %For $x \in \Lambda$, we write $t_x$ for the translation in $\widehat{S}_n$.

%Let $S_n$ be the symmetric group of order $n!$.  The Weyl group of $\GL_n(\F)$ is the group of \defn{extended affine permutations} $\hS_n = \Lambda_n \ltimes \S_n \simeq N_{\GL_n(\F)}(\widehat{T})/\widehat{T}$, whose elements can be thought of as bijections $\hw: \mathbb{Z} \to \mathbb{Z}$ such that $\hw(i+n)=\hw(i)+n$ and $\sum_{i=1}^n \hw(i) = \binom{n+1}{2} \mod n$. 

\begin{definition}
A \defn{Kostant partition} $(a_\alpha)_{\alpha \in \Phi^+}$ for $\lambda \in Q^+$ is a sequence of nonnegative integers indexed by positive roots such that $\lambda=\sum_{\alpha \in \Phi^+} a_\alpha \alpha$.  We denote the set of all Kostant partitions for $\lambda$ by $K(\lambda)$.
\end{definition}

Opdam proved the following formula for the trace in the Hecke algebra $\widehat{H}_n$ for the extended affine symmetric group.
\begin{theorem}[{\cite[Cor.~1.18]{opdam2003generating}}] Let $[k]_q=\frac{(q-1)^2}{q} \frac{q^k-q^{-k}}{q-q^{-1}}$.  For $\lambda=\lambda_+-\lambda_- \in Q^+$,
\[\mathrm{tr}(T_{{\lambda_-}}^{\phantom{-1}} T_{{\lambda_+}}^{-1}) = q^{(\ell({\lambda_-})-\ell({\lambda_+}))/2}\sum_{(a_\alpha) \in K(\lambda)} \prod_{\substack{\alpha \in \Phi^+ \\ a_\alpha > 0}} [a_\alpha]_q.\]
\label{thm:opdam}
\end{theorem}

%Let $\S_n$ act diagonally on the polynomial ring $\mathbb{C}[x_1,\ldots,x_n,y_1,\ldots,y_n]$, and write $Q_n$ for its root lattice, $\Lambda_n$ for its fundamental weights, and $\Phi^+_n$ for its positive roots.  The \emph{quotient ring of diagonal coinvariants} $\DH_n$ is the quotient of this polynomial ring by the ideal generated by the invariants with no constant term; there is a more general $\S_n$ module $\DH_n^m$ depending on an integral parameter $m$.  
\subsection{Haglund's identity}
In~\cite{haglund2011polynomial}, Haglund proved a remarkable formula for the bigraded (in $x$- and $y$-degree) Hilbert series of the quotient ring of diagonal coinvariants.  Haglund stated the formula in terms of \emph{Tesler matrices}, which are a simple combinatorial rephrasing of Kostant partitions.   %We choose to write the formula using Kostant partitions to mirror~\Cref{thm:opdam}---note that since there are only $n-1$ simple roots in $\Phi^+_n$, the formula is written using Kostant partitions in $Q_n^+$. %(m(n-2)+2)](n-1) 

\begin{theorem}[{\cite[Corollary 1]{haglund2011polynomial}}]\label{thm:Haglund} Write $[k]_{q,t} = (q-1)(1-t)\frac{q^k-t^k}{q-t}$ and
let $\lambda_n \coloneqq n\lambda_{n-1} \in Q^+_n$.  Then
\[\mathrm{Hilb}(\DH_{n-1};q,t) = \left(\frac{1}{(q-1)(t-1)}\right)^{n-1}\sum_{(a_\alpha) \in K(\lambda)} \prod_{\substack{\alpha \in \Phi^+_n \\ a_\alpha > 0}} [a_\alpha]_{q,t}.\]
\end{theorem}

%\begin{corollary}%[{\cite[]}]
%\[\mathrm{tr}(T_{\lambda_n}^{-1})=\sum_{(a_\alpha) \in K(\lambda)} \prod_{\substack{\alpha \in \Phi^+_n \\ a_\alpha > 0}} [a_\alpha]_{q} = (q - 1)^{2(n-1)}[n]_q^{n-2}.\]
%\end{corollary}

\subsection{Cyclic enumeration}
\label{sec:cyclic_enumeration}

\begin{theorem}[P.~Galashin, T.~Lam, N.~Williams]
\label{thm:subword_count}
\[
\left|R_{\blambda_n}(\FF_q) \right| =  (q - 1)^{2n-2}[n]_q^{n-2} \hspace{1em}\text{ and }\hspace{1em}  |\S_n|=n^{n-2}.
\] %Moreover,
%Let %$P_n = \mathcal{D}_{e,\w}$ be the set of subwords of $(s_0,s_1,\ldots,s_{n})^n$ of length $n(n-1)$ whose product is the identity and whose consecutive products decrease in weak order whenever possible.  
%\[|\S_n|=n^{n-2}.\]
\end{theorem}

\begin{proof}
Since $[k]_q = [k]_{q,q^{-1}}$, we can use Opdam's~\Cref{thm:opdam} and specialize Haglund's~\Cref{thm:Haglund} to conclude that \[\left|R_{{\blambda_n}}(\FF_q)\right| = q^{\ell(\lambda_n)} \tr(T_{\blambda_n^{-1}}) = (q - 1)^{2n-2}[n]_q^{n-2}.\]  Since all maximal distinguished subwords have exactly $2n-2$ skips and all other distinguished subwords have more than $2n-2$ skips, we have
\begin{align*}(q - 1)^{2n-2}[n]_q^{n-2}=\left|R_{{\blambda_n}}(\FF_q)\right| &= \sum_{\u \in \mathcal{D}_{\blambda_n}} \left|D_\u(\FF_q)\right| = \sum_{\u \in \S_n} \left|D_\u(\FF_q)\right| + \sum_{\u \not \in \S_n} \left|D_\u(\FF_q)\right| \\ &= \sum_{\u \in \S_n} (q-1)^{2n-2} q^{(n-1)(n-2)/2} + \sum_{\u \not \in \S_n}  (q-1)^{s(\u)} q^{t(\u)}, \end{align*}
where $s(\u)>2n-2$ for all $\u \not \in \S_n$. Dividing by $(q-1)^{2n-2}$ and letting $q \to 1$ gives $|\S_n|=n^{n-2}$.
% to conclude the following.
\end{proof}

\begin{corollary}[Cayley's formula]
\label{cor:cayley}
$|\E_n|=n^{n-2}.$
\end{corollary}
\begin{proof}
This follows immediately from~\Cref{thm:subword_count,thm:main_thm}.
\end{proof}

\section{Future Work}
\label{sec:future_work}

%\begin{itemize}
%\item Show that the product of the ``finite reflections'' gives a refinement according to the long cycle coinciding with the usual parking numbers (ex: $16=5+3+3+2+2+1$).  Actually, it seems like what you really want to do is remember just the root vertex $n$ and forget all the other labels!

%\item Is the number of factorizations of $\lambda_n^{-1}$ into $2n-2$ affine reflections of the form $(\!(i,j)\!)$ or $(\!(j-n,i)\!)$ for $1\leq i<j\leq n$ given by \href{https://oeis.org/A181167}{A181167}?  Note that not all of these use pairs of $(i,j)$ and $(j-n,i)$.  For example, $\lambda_4^{-1}=(\!(03)\!)(\!(\overline{1}2)\!)(\!(10)\!)(\!(1\overline{1})\!)(\!(\overline{2}1)\!)(\!(14)\!)$.

%[(3, 0), (2, -1), (1, 0), (1, -1), (1, -2), (1, 4)]

\subsection{Distinguished subwords}
It would be interesting to give a combinatorial interpretation for \emph{all} distinguished subwords of $\blambda_n$.  For $n=2,3,4,5$, the number of such subwords is $1,4,45,1331$; this sequence does not appear in the Online Encyclopedia of Integer Sequences.

\subsection{Other weights}
There should be a Fuss--Catalan extension~\cite{williamsopac}, using the translation \[\lambda_{m,n}=(m(n-1)+1)\lambda_{n-1}-(m-1)\lambda_1.\]  Maximal distinguished subwords will still be parameterized by trees, but the combinatorics of the run-leaves will be more complicated---the number of maximal subwords will be $(m(n-1)+1)^{n-2}$.

Much more generally~\cite[Conjecture 7.1]{armstrong2012combinatorics}, there should be interesting combinatorics coming from the weight \[\lambda = \sum_{i=1}^{n-1} a_i \alpha_i \text{ with } a_1 > a_2 > \cdots > a_{n-1} \geq a_n=0.\]  In this case, the number of maximal distinguished subwords is~\cite{armstrong2012combinatorics}
\[\prod_{i=1}^{n-1} \big((i+1)a_i-i a_{i+1}\big).\]

\subsection{Relation to Galashin-Lam-Trinh-Williams}

In this section we explore the possibility of a relationship between $R_{\blambda_n}(\FF_q)$ and the rational noncrossing parking functions (and their braid varieties) of~\cite[Section 8.5]{galashin2022rational}.

\begin{definition}
  Let $\u \in \S_n$.  We say that a skip in $\u$ is a \defn{negative} if the corresponding inversion $(\!(a,b)\!)$ in $\rru$ satisfies $a < b$ and $1 \leq (b \mod n) < (a \mod n) \leq n$.  A skip is \defn{positive} if it is not negative.
\end{definition}

\begin{example}
The negative skips are colored purple in~\Cref{fig:main,fig:distinguished}, while the positive skips are left in white.  Observe that there is exactly one negative skip in each column and each row except the last.
\end{example}

\begin{proposition}
\label{prop:negatives}
  Each $\u \in \S_n$ has exactly one negative skip in each column and each row except the last.%is written in a grid with $n$ rows of length $n-1$, the negative skips form a permutation matrix in the first $n-1$ rows.
\end{proposition}
\begin{proof}
  It follows from \Cref{nest} and \Cref{locSkip} that the pair of negative and positive skips $\rrs$, $\rrf$ appear in column $k$.
  Suppose that $t_1= (\!(\bar{a}_0,a_1)\!)$, where $1 \leq a_1 < a_0 \leq n$, is the inversion of the first negative skip in some row of our grid. The next skip has inversion $t_2 = (\!(a_1,a_2)\!)$ or $(\!(\bar{a}_1,a_2)\!)$, where $1 \leq a_2 \leq n$. If $t_2$ is in the same row as $t_1$, then its column number must be greater than $a_1$. It follows that $a_2 > a_1$, so $t_2 = (\!(a_1,a_2)\!)$. If the next skip is again in the same row, then its column number must be greater than $a_2$, so its inversion is $t_3 = (\!(a_2,a_3)\!)$, where $a_2 < a_3 \leq n$. Continuing in this way, we see that there cannot be another negative skip in this row.
  
  It remains to show that the last row of the grid cannot contain a negative skip. The last skip cannot be negative because its inversion is $(\!(a,n )\!)$ by \Cref{nest}. Moreover, every skip in the last row of the grid must be the second in its pair since the pairs occur in the same column. So if there is a negative skip with inversion $(\!(\bar{a}_0, a_1)\!)$ in the last row, it must be in column $a_0$ and the next skip must be in column $a_1$, which contradicts $a_1 < a_0$. So there cannot be a negative skip in the last row.
\end{proof}

We will not recall the definitions of the rational noncrossing parking braid varieties $R_{\mathsf{c}^{n+1}}^{(w)}(\FF_q)$ here, instead referring the interested reader to~\cite{galashin2022rational}.  We will simply describe how to use~\Cref{prop:negatives} to conjecturally break our braid variety $R_{\blambda_n}(\FF_q)$ into pieces that should match the individual components of the noncrossing parking braid varieties (indexed by $w \in S_n$).

\begin{remark}
Minh-T\^am Trinh has constructed certain braid variety variants that bundle together the individual parking braid variety components by enriching the usual definition of braid variety by elements of the unipotent subgroup of $B^+$.  There should be an isomorphism between $R_{\blambda_n}(\FF_q)$ and this variant for the symmetric group $S_{n-1}$ and the braid $\mathbf{c}^{n}$, where $\mathbf{c}$ is the lift of any standard Coxeter element in $S_{n-1}$ to its braid group.
\end{remark}

For $w \in S_{n-1}$, write $\S_n(w)$ for the set of all subwords $\u \in \mathcal{D}_{\blambda_n}$ with negative skips in the positions of the ones in the $(n-1)\times(n-1)$ permutation matrix of $w$, and takes in the positions corresponding to the inversions of $w$ (indices to the left of and above the ones).  Write \[R_{\blambda_n}^{(w)}(\FF_q) = \bigsqcup_{\u \in \S_n(w)} D_\u(\FF_q).\]

\begin{example}
If we fix $w=[1,3,2] \in S_3$, $\S_4(w)$ consists of all distinguished subwords that must use skips in the purple boxes and takes in the green boxes of \raisebox{0\height}{\scalebox{.5}{$\begin{array}{|C{2em}|C{2em}|C{2em}|}\hline \cd & & \\\hline & \cc & \cd \\\hline  &  \cd & \\\hline & & \\\hline \end{array}$}}.
Then $\S_4(w)$ contains three maximal distinguished subwords, and eight distinguished words in total:

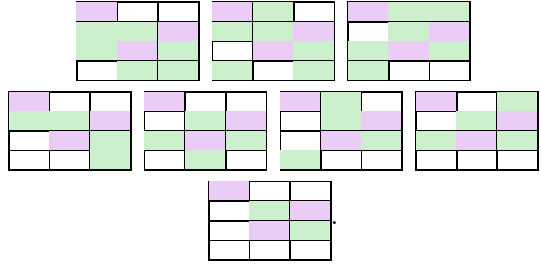
\[\scalebox{.5}{$\begin{array}{|C{2em}|C{2em}|C{2em}|}\hline
\cd & & \\\hline
\cc& \cc & \cd \\\hline
\cc&  \cd & \cc \\\hline
& \cc&\cc \\\hline \end{array}\hspace{1em}
\begin{array}{|C{2em}|C{2em}|C{2em}|}\hline
\cd & \cc & \\\hline
\cc & \cc & \cd \\\hline
&  \cd & \cc \\\hline
\cc& &\cc \\\hline \end{array}\hspace{1em}
\begin{array}{|C{2em}|C{2em}|C{2em}|}\hline
\cd & \cc&\cc \\\hline
& \cc & \cd \\\hline
\cc&  \cd &\cc \\\hline
\cc & & \\\hline \end{array}$}
\]
\[
\scalebox{.5}{$\begin{array}{|C{2em}|C{2em}|C{2em}|}\hline
\cd & & \\\hline
\cc& \cc & \cd \\\hline 
&  \cd & \cc\\\hline
& &\cc \\\hline \end{array} \hspace{1em}
\begin{array}{|C{2em}|C{2em}|C{2em}|}\hline
\cd & & \\\hline
& \cc & \cd \\\hline
\cc&  \cd &\cc \\\hline
& \cc& \\\hline \end{array}\hspace{1em}
\begin{array}{|C{2em}|C{2em}|C{2em}|}\hline
\cd &\cc & \\\hline
& \cc & \cd \\\hline
&  \cd &\cc \\\hline
\cc& & \\\hline \end{array}\hspace{1em}
\begin{array}{|C{2em}|C{2em}|C{2em}|}\hline
\cd & & \cc\\\hline
& \cc & \cd \\\hline
\cc&  \cd &\cc \\\hline
& & \\\hline \end{array}$}\]
\[
\scalebox{.5}{$\begin{array}{|C{2em}|C{2em}|C{2em}|}\hline
\cd & & \\\hline
& \cc & \cd \\\hline 
&  \cd &\cc \\\hline
& & \\\hline \end{array}$}.\]
%[[3, 4, 6, 8, 10, 11],
% [1, 3, 4, 8, 9, 11],
% [3, 4, 8, 11],
% [4, 6, 8, 10],
% [1, 2, 4, 6, 8, 9],
% [1, 4, 8, 9],
% [2, 4, 6, 8],
% [4, 8]]
\end{example}

\begin{conjecture}  We have a disjoint decomposition \[R_{\blambda_n}(\FF_q) = \bigsqcup_{w \in S_{n-1}} R_{\blambda_n}^{(w)}(\FF_q).\]
Moreover, for $w \in S_{n-1}$ and $\mathsf{c}_{n-1}=[s_1,\ldots,s_{n-2}]$, we have an isomorphism
\[R_{\blambda_n}^{(w)}(\FF_q) \simeq (\FF_q^\times)^{n-1} \times \FF_q^{\ell(w)} \times R_{\mathsf{c}_{n-1}^{n}}^{(w)}(\FF_q),\] where $R_{\mathsf{c}_{n-1}^{n}}^{(w)}(\FF_q)$ are the noncrossing parking braid varieties of~\cite{galashin2022rational}.
\end{conjecture}

\begin{remark}
When $w$ is the identity of $S_{n-1}$, the subwords in $\S_n(e)$ skip all instances of the affine reflection $s_0$ in $\blambda_n$ and there are no required takes (since the identity has no inversions).  Writing $\mathsf{c}_n=[s_1,\ldots,s_{n-1}]$, we immediately have \[R_{\blambda_n}^{(e)}(\FF_q) \simeq (\FF_q^\times)^{n-1} \times R_{\mathsf{c}_n^{n-1}}(\FF_q),\] where $R_{\mathsf{c}_n^{n-1}}(\FF_q)$ is the Fuss-Dogolon braid variety for $S_{n}$, which can easily be shown to be isomorphic to the usual Catalan braid variety $R_{\mathsf{c}_{n-1}^{n}}(\FF_q)$ in $S_{n-1}$. %as the braid variety $R_{}$ in $S_n$.
\end{remark}

\section*{Acknowledgements}
We thank Elise Catania, Sasha Pevzner, and Sylvester Zhang for organizing the 2023 Minnesota Research Workshop in Algebra and Combinatorics (MRWAC), where this work began.  We thank the Department of Mathematics at the University of Minnesota, Twin Cities, for providing excellent working conditions.  We acknowledge Lee Trent for participating in the early stages of the project.

This work benefited from computations in \texttt{Sage}~\cite{sagemath} and the combinatorics features developed by the \texttt{Sage-Combinat} community~\cite{Sage-Combinat}.

Esther Banaian was supported by Research Project 2 from the Independent Research Fund Denmark (grant no. 1026-00050B). Anh Trong Nam Hoang was supported by the University of Minnesota Doctoral Dissertation Fellowship. Elizabeth Kelley was supported by the National Science Foundation under Award No.\ 1937241. Weston Miller, Jason Stack, and Nathan Williams were partially supported by the National Science Foundation under Award No.\ 2246877.
\bibliographystyle{amsalpha}
\bibliography{cayley}

\newcommand{\etalchar}[1]{$^{#1}$}
\providecommand{\bysame}{\leavevmode\hbox to3em{\hrulefill}\thinspace}
\providecommand{\MR}{\relax\ifhmode\unskip\space\fi MR }
\providecommand{\MRhref}[2]{%
  \href{http://www.ams.org/mathscinet-getitem?mr=#1}{#2}
}
\providecommand{\href}[2]{#2}
\begin{thebibliography}{GLTW22}

\bibitem[AGH{\etalchar{+}}12]{armstrong2012combinatorics}
Drew Armstrong, Adriano Garsia, James Haglund, Brendon Rhoades, and Bruce
  Sagan, \emph{Combinatorics of {T}esler matrices in the theory of parking
  functions and diagonal harmonics}, Journal of Combinatorics \textbf{3}
  (2012), no.~3, 451--494.

\bibitem[BB05]{bjorner2005combinatorics}
Anders Bj{\"o}rner and Francesco Brenti, \emph{Combinatorics of {C}oxeter
  groups}, vol. 231, Springer, 2005.

\bibitem[BH21]{bao2021flag}
Huanchen Bao and Xuhua He, \emph{Flag manifolds over semifields}, Algebra \&
  Number Theory \textbf{15} (2021), no.~8, 2037--2069.

\bibitem[Cay89]{cayley}
Arthur Cayley, The Quarterly Journal of Mathematics \textbf{23} (1889),
  376--378.

\bibitem[Deo85]{deodhar1985some}
Vinay~V Deodhar, \emph{On some geometric aspects of {B}ruhat orderings. {I}. a
  finer decomposition of {B}ruhat cells}, Inventiones mathematicae \textbf{79}
  (1985), no.~3, 499--511.

\bibitem[GJ97]{goulden1997transitive}
Ian Goulden and David Jackson, \emph{Transitive factorisations into
  transpositions and holomorphic mappings on the sphere}, Proceedings of the
  American Mathematical Society \textbf{125} (1997), no.~1, 51--60.

\bibitem[GLTW22]{galashin2022rational}
Pavel Galashin, Thomas Lam, Minh-T{\^a}m~Quang Trinh, and Nathan Williams,
  \emph{Rational noncrossing {C}oxeter-{C}atalan combinatorics}, arXiv preprint
  arXiv:2208.00121 (2022).

\bibitem[GY02]{goulden2002tree}
Ian Goulden and Alexander Yong, \emph{Tree-like properties of cycle
  factorizations}, Journal of Combinatorial Theory, Series A \textbf{98}
  (2002), no.~1, 106--117.

\bibitem[Hag11]{haglund2011polynomial}
James Haglund, \emph{A polynomial expression for the {H}ilbert series of the
  quotient ring of diagonal coinvariants}, Advances in Mathematics \textbf{227}
  (2011), no.~5, 2092--2106.

\bibitem[KL79]{kazhdan1979representations}
David Kazhdan and George Lusztig, \emph{Representations of {C}oxeter groups and
  {H}ecke algebras}, Inventiones mathematicae \textbf{53} (1979), no.~2,
  165--184.

\bibitem[Lei21]{leinster2021probability}
Tom Leinster, \emph{The probability that an operator is nilpotent}, The
  American Mathematical Monthly \textbf{128} (2021), no.~4, 371--375.

\bibitem[LMPS19]{lewis2019computing}
Joel Lewis, Jon McCammond, Kyle Petersen, and Petra Schwer, \emph{Computing
  reflection length in an affine {C}oxeter group}, Transactions of the American
  Mathematical Society \textbf{371} (2019), no.~6, 4097--4127.

\bibitem[MP11]{mccammond2011bounding}
Jon McCammond and T~Kyle Petersen, \emph{Bounding reflection length in an
  affine {C}oxeter group}, Journal of Algebraic Combinatorics \textbf{34}
  (2011), 711--719.

\bibitem[Opd03]{opdam2003generating}
Eric~M Opdam, \emph{A generating function for the trace of the
  {I}wahori--{H}ecke algebra}, Springer, 2003.

\bibitem[SCc08]{Sage-Combinat}
The {S}age-{C}ombinat community, \emph{{S}age-{C}ombinat: enhancing {S}age as a
  toolbox for computer exploration in algebraic combinatorics}, 2008,
  \texttt{http://combinat.sagemath.org}.

\bibitem[Sta97]{stanley1997parking}
Richard~P Stanley, \emph{Parking functions and noncrossing partitions}, The
  Electronic Journal of Combinatorics [electronic only] \textbf{4} (1997),
  no.~2, v4i2r20--pdf.

\bibitem[{The}21]{sagemath}
{The Sage Developers}, \emph{{S}agemath, the {S}age {M}athematics {S}oftware
  {S}ystem ({V}ersion 9.4)}, 2021, \texttt{https://www.sagemath.org}.

\bibitem[Wil23]{williamsopac}
Nathan Williams, \emph{Combinatorics of braid varieties}, Proceedings of
  Symposia in Pure Mathematics (2023+).

\end{thebibliography}
\end{document}